\documentclass[a4paper,11pt]{article}

\usepackage[T1]{fontenc}
\usepackage{lmodern}

\usepackage{dsfont}

\usepackage[latin1]{inputenc}
\usepackage{amsmath}
\usepackage{amsthm}
\usepackage{amssymb}
\usepackage{mathrsfs}
\usepackage{graphicx}
\usepackage[all]{xy}
\usepackage{hyperref}

\usepackage{makeidx}

\usepackage{stmaryrd}
\usepackage{caption}

\usepackage{abstract} 

\newtheorem{thm}{Theorem}[section]
\newtheorem{cor}[thm]{Corollary}
\newtheorem{claim}[thm]{Claim}
\newtheorem{fact}[thm]{Fact}

\newtheorem{lemma}[thm]{Lemma}
\newtheorem{prop}[thm]{Proposition}

\theoremstyle{definition}
\newtheorem{definition}[thm]{Definition}

\newtheorem{question}[thm]{Question}

\def\rquotient#1#2{%
	\makeatletter
	\raise.3ex\hbox{$#1$}/\lower.3ex\hbox{$#2$}%
	\makeatother
}	

\makeatletter
\newcommand{\subjclass}[2][2010]{%
	\let\@oldtitle\@title%
	\gdef\@title{\@oldtitle\footnotetext{#1 \emph{Mathematics subject classification.} #2}}%
}
\newcommand{\keywords}[1]{%
	\let\@@oldtitle\@title%
	\gdef\@title{\@@oldtitle\footnotetext{\emph{Key words and phrases.} #1.}}%
}
\makeatother

\newcommand{\Address}{{
		\bigskip
		\small
		
\noindent\textsc{Universit\'e de Montpellier\\ 
Institut Math\'ematiques Alexander Grothendieck\\
Place Eug\`ene Bataillon\\
34090 Montpellier (France)}\par\nopagebreak
\noindent \textit{E-mail address}: \texttt{anthony.genevois@umontpellier.fr}
		
}}

\makeindex

\title{Relative numbers of ends and quasi-median graphs}
\date{\today}
\author{Anthony Genevois}


\begin{document}

\maketitle

\begin{abstract}
Given a finitely generated $G$ and a subgraph $H \leq G$, the \emph{relative number of ends} $e(G,H)$ is the number of ends of a Schreier graph $\mathrm{Sch}(G,H)$ and the \emph{number of coends} $\tilde{e}(G,H)$ is the maximal number of $H$-infinite components of the complement of a neighbourhood of $H$ in $G$. Generalising Sageev's characterisation of codimension-one subgroups in terms of actions on CAT(0) cube complexes, we characterise the number of relative ends and the number of coends of a pair $(G,H)$ in terms of actions on quasi-median graphs.
\end{abstract}

\tableofcontents

\section{Introduction}

\noindent
After Stallings' seminal theorem about ends of groups, there have been an interest in generalising the the number of ends of a finitely generated group to a number of ends relative to a given subgroup and to relative this notion to splittings. See the survey \cite[Chapter~6.1]{MR2031877} for more information on the subject. 

\medskip \noindent
In this article, given a finitely generated group $G$ and a subgroup $H \leq G$, we focus on two specific notions of relative ends. The first one, initially introduced in \cite{MR357679} (see also \cite{MR487104}), compute the classical number of ends $e(G,H)$ of a Schreier graph $\mathrm{Sch}(G,H)$. The second one, introduced independently in several places and under different names \cite{MR1025923,MR2365352,MR1867372}, compute the maximal (possibly infinite) number $\tilde{e}(G,H)$ of $H$-infinite connected components in the complement of a neighbourhood of $H$ in $G$. It is known that the inequality $\tilde{e}(G,H) \geq e(G,H)$ always holds, but otherwise the two quantities can be quite different. 

\medskip \noindent
If $e(G,H) \geq 2$, one says that $H$ is a \emph{codimension-one subgroup}; and, if $\tilde{e}(G,H) \geq 2$, one says that $H$ is a coarsely separating subgroup. It is known that, if our group $G$ splits non-trivially over $H$, then $H$ is a codimension-one subgroup, and a fortiori a coarsely separating subgroup. A problem that received a lot of attention is to which extent the converse may be satisfied.

\medskip \noindent
Geometrically speaking, it follows from Stallings' theorem that $G$ has $\geq 2$ ends if and only if it admits an action on a tree with unbounded orbits, without edge-inversions, and with finite edge-stabilisers. In \cite{MR1347406}, Sageev generalised this statement by proving that $e(G,H) \geq 2$ if and only if $G$ admits a specific action on a median graph (or equivalently, a CAT(0) cube complex\footnote{Despite the fact that median graphs and CAT(0) cube complexes define the same objects, we choose to speak about median graphs, as justified in \cite{WhyMedian}.}) with $H$ as a hyperplane-stabiliser. Thus, it is possible to detect geometrically when a subgroup has codimension one. But, then, a natural question is: can we detect geometrically the number of relative ends $e(G,H)$? And what about the number of coends $\tilde{e}(G,H)$?

\medskip \noindent
In this article, we show how the geometry of quasi-median graphs can be used in order to answer these questions. More precisely, the main result of this paper is:

\begin{thm}\label{thm:BigIntro}
Let $G$ be a finitely generated group and $H \leq G$ a subgroup. For all $p \in \mathbb{N}$ and $q \in \mathbb{N}_{\geq 2} \cup \{ \omega \}$, the following assertions are equivalent:
\begin{itemize}
	\item $e(G,H) \geq p$ and $\tilde{e}(G,H) \geq q$;
	\item $G$ admits an $H$-monohyp action on a quasi-median graph such that a hyperplane stabilised by $H$ delimits $\geq q$ sectors and $\geq p$ $H$-orbits of sectors;
	\item $G$ admits an $H$-monohyp action on a Hamming graph such that a hyperplane stabilised by $H$ delimits $\geq q$ sectors and $\geq p$ $H$-orbits of sectors.
\end{itemize}
\end{thm}

\noindent
Loosely speaking, a quasi-median graph is a median graph in which each hyperplane disconnects the graph into at least two, but possibly more, connected components, called \emph{sectors}. Beyond this difference, the geometry of quasi-median graphs is basically the same as the geometry of median graphs. Formally, quasi-median graphs can be defined as retracts of Hamming graphs (i.e.\ product of complete graphs), mimicking the characterisation of median graphs as retracts of hypercubes; or as one-skeleta of CAT(0) prism complexes (where prisms refer to products of simplices). 

\medskip \noindent
Since Theorem~\ref{thm:BigIntro} does not only characterise codimension-one and coarsely separating subgroups, but also computes the number of relative ends and the number of coends, it is not sufficient to deal with ``non-trivial'' actions on quasi-median graphs. The actions has to be ``minimal'' in some sense. This justifies the introduction of \emph{monohyp} actions.

\begin{definition}
Let $G$ be a group acting on a quasi-median graph $X$. The action is \emph{monohyp} if 
\begin{itemize}
	\item it has unbounded orbits,
	\item it is \emph{hyperplane-transitive}, i.e.\ $X$ contains a single $G$-orbit of hyperplanes, and
	\item it is \emph{convex-minimal}, i.e.\ $X$ is the only non-empty $G$-invariant convex subgraph.
\end{itemize}
Given a subgroup $H \leq G$, the action is \emph{$H$-monohyp} if it is monohyp and $H$ is a hyperplane-stabiliser. 
\end{definition}

\noindent
We emphasize that convex-minimality is not a strong restriction since every action of a finitely generated group on a quasi-median graph can be easily turned into a convex-minimal action. See Proposition~\ref{prop:ConvexMinimal} below. 

\medskip \noindent
As particular cases of Theorem~\ref{thm:BigIntro}, one deduces that:

\begin{cor}\label{cor:BigIntro}
Let $G$ be a finitely generated group and $H \leq G$ a subgroup. The following assertions hold:
\begin{itemize}
	\item $e(G,H) \geq p$ if and only if $G$ admits an $H$-monohyp action on a quasi-median graph such that a hyperplane stabilised by $H$ delimits $\geq p$ $H$-orbits of sectors. 
	\item $H$ has codimension one if and only if $G$ admits an action on a median graph with unbounded orbits, without hyperplane-inversion, with a single orbit of hyperplanes, with hyperplane-stabilisers conjugate to $H$.
	\item $\tilde{e}(G,H) \geq q$ if and only if $G$ admits an $H$-monohyp action on a quasi-median graph such that a hyperplane stabilised by $H$ delimits $\geq q$ sectors. 
	\item $H$ is coarsely separating if and only if $G$ if and only if $G$ admits an action on a quasi-median graph with unbounded orbits, with a single orbit of hyperplanes, with hyperplane-stabilisers conjugate to $H$.
\end{itemize}
\end{cor}

\noindent
Notice that the second item recovers \cite{MR1347406,MR1459140}.

\medskip \noindent
We conclude the article with a discussion on the difference between codimension-one subgroups and coarsely separating subgroups, which we summarise in our next statement:

\begin{thm}\label{thm:IntroSep}
Let $G$ be a finitely generated group and $H \leq G$ a subgroup. 
\begin{itemize}
	\item If $H$ is codimension-one then it is coarsely separating.
	\item If $2 \leq \tilde{e}(G,H) \leq \omega$ is finite, then $H$ is virtually codimension-one.
	\item If $H$ is an ERF group, then $H$ is coarsely separating if and only if it is virtually codimension-one.
	\item If $H$ is coarsely separating, then it contains a codimension-one subgroup of $G$.
	\item If $H$ is coarsely separating, then it may have no finite-index subgroup that has codimension-one in $G$.
\end{itemize}
\end{thm}

\noindent
Recall that a subgroup $H \leq G$ is \emph{separable} if, for every $g \in G \backslash H$, there exists a finite-index subgroup in $G$ that contains $H$ but not $g$. A group all of whose subgroups are separable is \emph{ERF} (for \emph{Extended Residually Finite}). Examples of ERF groups include virtually polycyclic groups (see \cite[Statement~1.3.10]{MR2093872}).

\paragraph{A few words about the proof of Theorem~\ref{thm:BigIntro}.} Let $G$ be a finitely generated group and $H \leq G$ a subgroup. First, assume that $G$ admits an $H$-monohyp action on a quasi-median graph $X$. In order to estimate the relative numbers of ends of $H$, we fix a hyperplane $J$ of $X$ stabilised by $H$, and, for every sector delimited by $J$, we consider the set
$$A_S:= \{ g \in G \mid g \cdot o \in V(S)\}$$
where $o \in V(X)$ is an arbitrary basepoint. Following \cite{MR1347406}, we know that, when $X$ is a median graph, $A_S$ is an \emph{almost invariant subset}, which allows us to conclude that $H$ is a codimension-one subgroup. When $X$ is a quasi-median graph, one can show similarly that at least one $A_S$ is an almost invariant subset. But The difficult here is to show that all the $A_S$ are almost invariant subsets. Once this is known, then we can conclude easily as the relative numbers of ends are related to the number of pairwise disjoint almost invariants subsets; see Propositions~\ref{prop:AlmostInCoarseSep} and~\ref{prop:AlmostInvCodim}. 

\medskip \noindent
In order to overcome the difficulty, we need the action $G \curvearrowright X$ to be more than just convex-minimal. In Section~\ref{section:GraphPrisms}, we show that every quasi-median graph $X$ has a median model $\mathbb{P}(X)$, its \emph{graph of prisms}. What we need, is that the action $G \curvearrowright X$ as well as its induced action $G \curvearrowright \mathbb{P}(X)$ are both convex-minimal. According to Proposition~\ref{prop:WhenStrongConvMin}, this amounts to saying that every prism in $X$ has a $G$-translate in every sector of $X$. In general, the convex-minimality of $G \curvearrowright X$ does not guarantee the convex-minimality of $G \curvearrowright \mathbb{P}(X)$. However, we prove in Section~\ref{section:StrongConvMin} that $G \curvearrowright \mathbb{P}(X)$ is convex-minimal whenever $G \curvearrowright \mathbb{P}(X)$ is monohyp. 

\medskip \noindent
Conversely, assume that $H$ is a coarsely separating subgroup. In other words, there exists an $L \geq 0$ such that $G\backslash H^{+L}$ has $H$-infinite components. Consider 
$$\mathscr{C}:= \{H\text{-infinite components of } G\backslash H^{+L} \} \cup \{H, H_+\backslash H\}$$
where $H_+$ denotes the complement of the union of the $H$-infinite components of $G\backslash H^{+L}$. This defines a partition of $G$ whose stabiliser is $H$, from which we can define a $G$-invariant collection of partitions $\mathfrak{C}:=\{g \mathscr{C}, g \in G\}$. Generalising the cubulation of wallspaces, which produces median graphs, we show in Section~\ref{section:QuasiCubulation} how to construct quasi-median graphs from such a data: the \emph{quasi-cubulation} of \emph{spaces with characters}. (In Appendix~\ref{section:Buneman}, we give some background on this construction, mentioning some other constructions one can find, for instance, in phylogenetics.) Thus, from the data $(G,\mathfrak{C})$, we obtain a quasi-median graph $\mathrm{QM}(G,\mathfrak{C})$ on which $G$ acts. By turning this action into a convex-minimal action, we obtain an $H$-monohyp action of $G$ on some quasi-median graph.

\paragraph{Organisation of the article.} In Section~\ref{section:Prel}, we record some basic definitions and properties related to quasi-median graphs (Section~\ref{section:QM}) and relative numbers of ends (Section~\ref{section:Ends}). 

\medskip \noindent
In Section~\ref{section:QMtoM}, we introduce and study two median graphs associated to quasi-median graphs: the graph of polytopes and the graph of prisms. The graph of polytopes is a median graph that contains the graph of prisms, and which we use as a tool in this article but which is of independent interest. In Section~\ref{section:BigMin}, we study convex-minimal actions on quasi-median graphs (Section~\ref{section:Min}) and when the induced action on the graph of prisms remains convex-minimal (Section~\ref{section:StrongConvMin}). The main result there is that monohyp actions on quasi-median graphs induce convex-minimal actions on the graph of prisms (Theorem~\ref{thm:MonoHypConvMin}). With these tools in hand, we prove one implication of Theorem~\ref{thm:BigIntro} in Section~\ref{section:FromQM} by estimating the numbers of relative ends from a monohyp action on a quasi-median graph (Theorem~\ref{thm:FromQM}). 

\medskip \noindent
In Section~\ref{section:QuasiCubulation}, we introduce \emph{spaces with characters}, which are essentially sets endowed with collections of partitions, and we show how construct quasi-median graphs from such a data. This allows us to prove the other implication of Theorem~\ref{thm:BigIntro} in Section~\ref{section:FromCoarseSep} by contructing a monohyp action on a quasi-median graph from a coarsely separating subgroup (Theorem~\ref{thm:FromCoarseSep}).

\medskip \noindent
We compile the proof of Theorem~\ref{thm:BigIntro} and its consequences in Section~\ref{section:Proof}. In Section~\ref{section:CodimVSsep}, we discuss the difference between codimension-one subgroup and coarsely separating subgroups, and we prove Theorem~\ref{thm:IntroSep}. Finally, in Appendix~\ref{section:Buneman}, we give some background on the quasi-cubulation of spaces with characters, mentioning some other constructions one can find, for instance, in phylogenetics.


\section{Preliminaries}\label{section:Prel}

\subsection{Quasi-median graphs}\label{section:QM}

\noindent
Quasi-median graphs can be thought of as defining the smallest reasonable family of median-like graphs that includes median graphs and Hamming graphs (i.e.\ product of complete graphs). Formally, it is possible to define quasi-median graphs as retracts of Hamming graphs (referring to the characterisation of median graphs as retracts of hypercubes). But there exist many possible equivalent definitions, including a local-to-global characterisation as given below (Theorem~\ref{thm:LocalGlobal}). We refer the reader to \cite{MR1297190} for more information. 

\medskip \noindent
In this article, we define quasi-median graphs as particular \emph{weakly modular} graphs, as defined below. This is not the most intuitive definition, nor the easiest to digest, but it is rather efficient to recognise some basic examples and to extract some useful information about quasi-median graphs. 

\begin{definition}
A connected graph $X$ is \emph{weakly modular} if it satisfies the following condition:
\begin{description}
	\item[(Triangle Condition)] for every vertex $o \in V(X)$ and all adjacent vertices $x,y \in V(X)$ such that $d(o,x)=d(o,y)$, there exists a common neighbour $w \in V(X)$ of $x$ and $y$ such that $d(o,w)=d(o,x)-1$;
	\item[(Quadrangle Condition)] for all vertices $o,x,y \in V(X)$ and every common neighbour $z \in V(X)$ of $x$ and $y$ such that $d(o,z)=d(o,x)+1=d(o,y)+1$ and $d(x,y)=2$, there exists a common neighbour $w \in V(X)$ of $x$ and $y$ such that $d(o,w)=d(o,z)-2$. 
\end{description}
A \emph{quasi-median} graph is a weakly modular graph with no induced copy of the complete bipartite graph $K_{2,3}$ (= two $4$-cycles glued along two consecutive vertices) and $K_4^-$ (= the complete graph $K_4$ minus an edge, or equivalent two $3$-cycles glued along a common edge). 
\end{definition}
\begin{figure}[h!]
\begin{center}
\includegraphics[width=0.7\linewidth]{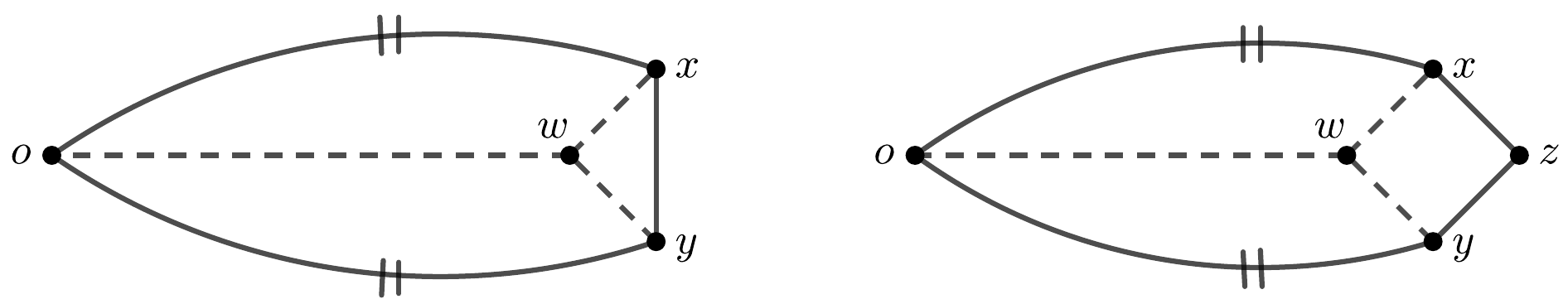}
\caption{Triangle and Quadrangle Conditions defining weakly modular graphs.}
\label{TQConditions}
\end{center}
\end{figure}

\noindent
Examples of quasi-median graphs include, as already said, median graphs and Hamming graphs, but also block graphs for instance. On the other hand, cycles of length $\geq 5$ are never quasi-median.

\paragraph{Gated subsets.} Recall that, given a graph $X$, a subgraph $Y \leq X$ is \emph{gated} if, for every $x \in V(X)$, there exists a vertex $y \in V(Y)$, referred to as the \emph{gate}, such that $x$ can be connected to every vertex in $Y$ by some geodesic that passes through $y$; or, for short, $Y \cap \bigcap_{y \in V(Y)} I(x,y) \neq \emptyset$, where $I(a,b):= \{ c \in V(X) \mid d(a,b)=d(a,c)+d(c,b)\}$ denotes the \emph{interval} between two vertices. Notice that, when it exists, the gate of $x$ in $Y$ coincides with the unique vertex of $Y$ that minimises the distance to $x$, justifying that the gate is also sometimes referred to as the \emph{projection of $x$ to $Y$}. 

\medskip \noindent
Gated subgraphs are always convex, but the converse may not hold. For instance, in a complete graph with $\geq 3$ vertices, an edge is convex but not gated. In quasi-median graphs, gated subgraphs will play the role of convex subgraphs in median graphs. (Notice that gated subgraphs define an abstract convexity, as defined in \cite{MR1234493}. So gatedness can be thought of as a form of convexity.) In median graphs, a subgraph is gated if and only if it is convex. 

\medskip \noindent
Gated subgraphs are known to satisfy the following Helly property: a finite collection of pairwise intersecting gated subgraphs always has non-empty global intersection. This observation will be quite useful in the article.

\paragraph{Hyperplanes.} Similarly to median graphs, the geometry of quasi-median graphs is encoded into the combinatorics of specific separating subgraphs, called \emph{hyperplanes}. 

\begin{definition}
Let $X$ be a quasi-median graph. A \emph{hyperplane} is an equivalence class of edges with respect to the reflexive-transitive closure of the relation that identify two edges whenever they belong to a common $3$-cycle or they are opposite edges in a $4$-cycle. The \emph{carrier} $N(J)$ of a hyperplane $J$ is the subgraph of $X$ induced by the edges in $J$. Its \emph{fibres} are the connected components of the graph $N(J) \backslash \backslash J$ obtained from $N(J)$ by removing the edges in $J$. Two hyperplanes $J_1$ and $J_2$ are \emph{transverse} if there exists an induced $4$-cycle whose two pairs of opposite edges respectively belong to $J_1$ and $J_2$. 
\end{definition}
\begin{figure}[h!]
\begin{center}
\includegraphics[width=0.5\linewidth]{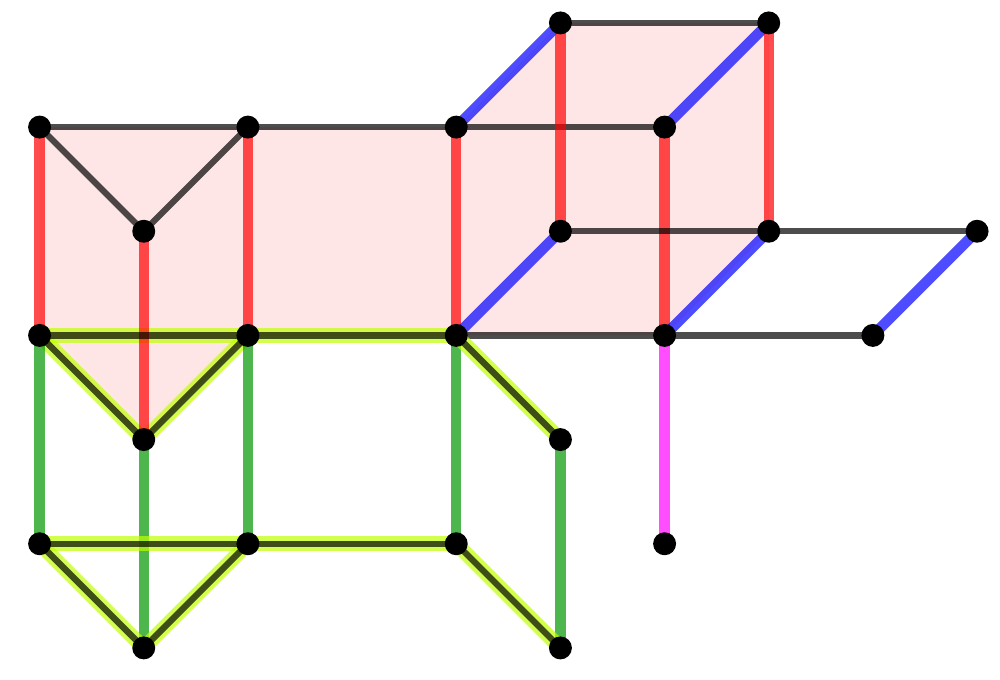}
\caption{A quasi-median graph and some of its hyperplanes. In red, a hyperplane and its carrier. In green, a hyperplane and its two fibres. The red and blue hyperplanes are transverse.}
\label{HyperplaneQM}
\end{center}
\end{figure}

\noindent
The connection between the geometry of a quasi-median graph and its hyperplanes is recorded in our next statement.

\begin{thm}[{\cite[Propositions~2.15 and~2.30]{QM}}]
Let $X$ be a quasi-median graph.
\begin{itemize}
	\item Every hyperplane $J$ separates $X$, i.e.\ the graph $X\backslash\backslash J$ obtained from $X$ by removing the edges from $J$ contains $\geq 2$ connected components. We call these components \emph{sectors}. 
	\item Sectors, carriers, and fibres of hyperplanes are gated.
	\item A path is geodesic if and only if it crosses each hyperplane at most once.
	\item The distance between two vertices coincides with the number of hyperplanes separating them.
\end{itemize}
\end{thm}

\noindent
The main difference between median and quasi-median graphs is that, in quasi-median graphs, hyperplanes delimit at least two sectors (possibly infinitely many), while, in median graphs, hyperplanes always delimit exactly two halfspaces. In fact, this property characterises median graphs among quasi-median graphs:

\begin{lemma}\label{lem:WhenMedian}
A quasi-median graph is median if and only if all its hyperplanes delimit exactly two sectors. 
\end{lemma}

\begin{proof}
Let $X$ be a quasi-median graph, $C$ a clique (i.e.\ a maximal complete subgraph), and $J$ a hyperplane containing $C$. It follows from \cite[Corollary~2.21]{QM} that the number of sectors delimited by $J$ coincides with the number of vertices of $C$. Therefore, all the hyperplanes of $X$ delimit exactly two sectors if and only if $X$ is triangle-free. This is a characterisation of median graphs according to \cite[(25) p.\ 149]{MR605838}. 
\end{proof}

\noindent
We conclude this paragraph by recording two technical lemmas that will be useful later.

\begin{lemma}[{\cite[Lemma~2.34]{QM}}]\label{lem:SepProj}
Let $X$ be a quasi-median graph and $Y \subset V(X)$ a gated subset. for every $x \in V(X)$, a hyperplane separating $x$ from its gate in $Y$ always separates $x$ from $Y$. 
\end{lemma}

\begin{lemma}\label{lem:InclusionOne}
Let $X$ be a quasi-median graph, $S$ a sector, and $Y_1, \ldots, Y_n$ pairwise intersecting gated subgraphs. If $Y_1 \cap \cdots \cap Y_n \leq S$ then $Y_i \leq S$ for some $1 \leq i \leq n$. 
\end{lemma}

\begin{proof}
Let $J$ denote the hyperplane delimiting $S$ and fix another sector $S'$ delimited by $J$. Given an index $1 \leq i \leq n$, we claim that, if $Y_i$ is not contained in $S$, then $Y_i$ intersects $S'$.

\medskip \noindent
We already know that $Y_i$ intersects $S$, so, if $Y_i$ is not contained in $S$, necessarily it must be crossed by $J$. Let $\{a,b\}$ be an edge of $Y_i$ that belongs to $J$. The clique of $X$ containing $\{a,b\}$ must be contained in $Y_i$, since $Y_i$ is gated; but it also intersects all ther sectors delimited by $J$, including $S'$. So we conclude that $Y_i$ intersects $S'$, as claimed.

\medskip \noindent
It follows that, if $Y_i$ is not contained in $S$ for every $1 \leq i \leq n$, then the gated subgraphs $S', Y_1, \ldots, Y_n$ pairwise intersect, which implies that $Y_1 \cap \cdots \cap Y_n \cap S'$ is non-empty, contradicting the fact that $Y_1 \cap \cdots \cap Y_n$ is contained in $S$ (and a fortiori disjoint from~$S'$). 
\end{proof}

\paragraph{Multisectors.} In addition to sectors, unions of sectors also play an important role in the geometry of quasi-median graphs. 

\begin{definition}
Let $X$ be a quasi-median graph. A \emph{multisector} is the subgraph induced by a union of sectors delimited by a given hyperplane. A \emph{cosector} is the complement of a sector.
\end{definition}

\noindent
Notice that sectors and cosectors are multisectors. 

\medskip \noindent
In some sense, sectors encode gatedness in quasi-median graphs. As an illustration of this phenomenon, we can mention that sectors are gated and that every gated subgraph coincides with the intersection of the sectors containing it. Cosectors play a similar role for convexity:

\begin{prop}[{\cite[Proposition~2.104 and Lemma~2.105]{QM}}]\label{prop:MultiSectorConvex}
Let $X$ be a quasi-median graph. Every multisector is convex. Moreover, every convex subgraph coincides with the intersection of the cosectors containing it. 
\end{prop}

\noindent
It is worth mentioning that the only gated multisectors are the sectors. More precisely:

\begin{prop}[{\cite[Lemma~2.102]{QM}}]\label{prop:GatedHullMulti}
Let $X$ be a quasi-median graph and $J$ a hyperplane. If $S_1, \ldots, S_n$ are pairwise distinct sectors delimited by $J$ and $n \geq 2$, then the gated hull of $S_1 \cup \cdots \cup S_n$ is $S_1 \cup \cdots \cup S_n \cup N(J)$. 
\end{prop}

\paragraph{Prisms.} In a quasi-median graph, a \emph{clique} is a maximal complete subgraph. A \emph{prism} is a subgraph that decomposes as a product of finitely many cliques. The number of cliques in this product-decomposition is the \emph{cubical dimension} of the prism. According to \cite[Lemmas~2.16 and~2.80]{QM}, cliques and prisms are gated.

\medskip \noindent
Cliques (resp.\ prisms) in quasi-median graphs play the role of edges (resp.\ cubes) in median graphs. For instance, every quasi-median graph can be endowed with a contractible (and even CAT(0)) cellular structure by filling the prisms with products of simplices. Prisms also appear naturally in fixed-point properties:

\begin{prop}\label{prop:FixedPointMedian}
Let $G$ be a group acting on a quasi-median graph $X$. If $G$ has bounded orbits, then there is a prism that is stabilised. If moreover $G$ acts on $X$ without hyperplane-inversion, then $G$ fixes a vertex.
\end{prop}

\noindent
Recall that a \emph{hyperplane-inversion} is an isometry that stabilises a hyperplane and permutes non-trivially the sectors it delimits. 

\begin{proof}[Proof of Proposition~\ref{prop:FixedPointMedian}.]
The first assertion is \cite[Theorem~2.115]{QM}. Now, assume that $G$ stabilises a prism $P$ and acts on $X$ without hyperplane-inversion. If $P$ is a single vertex, there is nothing to prove. Otherwise, there exists a hyperplane $J$ crossing $P$. Fix a sector $J^+$ delimited by $J$. Because $G$ does not contain any hyperplane-inversion, the (finitely many) sectors in the orbit $G \cdot J^+$ pairwise intersect. Because prisms and sectors are gated, it follows that 
$$P_0 : = P \cap \bigcap\limits_{g \in G} gJ^+$$
defines is non-empty. Thus, we have found a $G$-invariant prism $P_0$ of smaller cubical dimension than $P$. After finitely many iterations of this argument, we eventually find that $G$ fixes a vertex, as desired. 
\end{proof}

\noindent
We conclude this paragraph with a technical lemme that will be useful later. 

\begin{lemma}\label{lem:ConstructPrism}
Let $X$ be a quasi-median graph, $x \in V(X)$ a vertex, and $J_1, \ldots, J_n$ a collection of pairwise transverse hyperplanes. If $x \in N(J_1) \cap \cdots\cap N(J_n)$, then there exists a prism containing $x$ whose hyperplanes are exactly $J_1, \ldots, J_n$. 
\end{lemma}

\begin{proof}
For every $1 \leq i \leq n$, fix a sector $J_i^+$ delimited by $J_i$ that does not contain $x$. Since the $J_i$ are pairwise transverse, the $J_i^+$ pairwise intersect, so the intersection $I:= J_1^+ \cap \cdots \cap J_n^+$ is non-empty. Let $y$ denote the gate of $x$ in $I$. We claim that the hyperplanes separating $x$ and $y$ are exactly the $J_i$.

\medskip \noindent
It is clear that the $J_i$ all separate $x$ and $y$. Conversely, let $J$ be a hyperplane separating $x$ and $y$. According to Lemma~\ref{lem:SepProj}, $J$ separates $x$ from $I$; and, according to Lemma~\ref{lem:InclusionOne}, $J$ must separate $x$ from $J_i^+$ for some $1 \leq i \leq n$. Since $x \in N(J_i)$, the only possibility is that $J=J_i$. This concludes the proof of our claim.

\medskip \noindent
Let $P$ denote the gated hull of $\{x,y\}$. According to \cite[Proposition~2.68]{QM}, the hyperplanes of $P$ coincide with the hyperplanes of $X$ separating $x$ and $y$, namely $J_1, \ldots, J_n$. In particular, the hyperplanes of $P$ pairwise intersect. It follows from \cite[Lemma~2.74]{QM} that $P$ is a prism. This is the prism we were looking for. 
\end{proof}

\paragraph{Local-to-global characterisation.} In order to recognise the geometry of a graph (or of a cellular complex), it may be desirable to have a local characterisation that is easy to check. Below is such a criterion for quasi-median graphs, which extends a well-known criterion for median graphs:

\begin{thm}\label{thm:LocalGlobal}
A connected graph $X$ is quasi-median if and only if it satisfies the following conditions:
\begin{itemize}
	\item the square-triangle-completion $X^{\square\triangle}$ is simply connected;
	\item $X$ does not contain induced copies of $K_{2,3}$ or $K_4^-$;
	\item $X$ satisfies the \emph{$3$-cube condition}, i.e.\ every induced copy of $Q_3^-$ is contained in a $3$-cube $Q_3$;
	\item $X$ satisfies the \emph{$3$-prism condition}, i.e.\ every induced copy of the house graph is contained in a $3$-prism $K_2 \times K_3$. 
\end{itemize}
\end{thm}
\begin{figure}[h!]
\begin{center}
\includegraphics[width=0.8\linewidth]{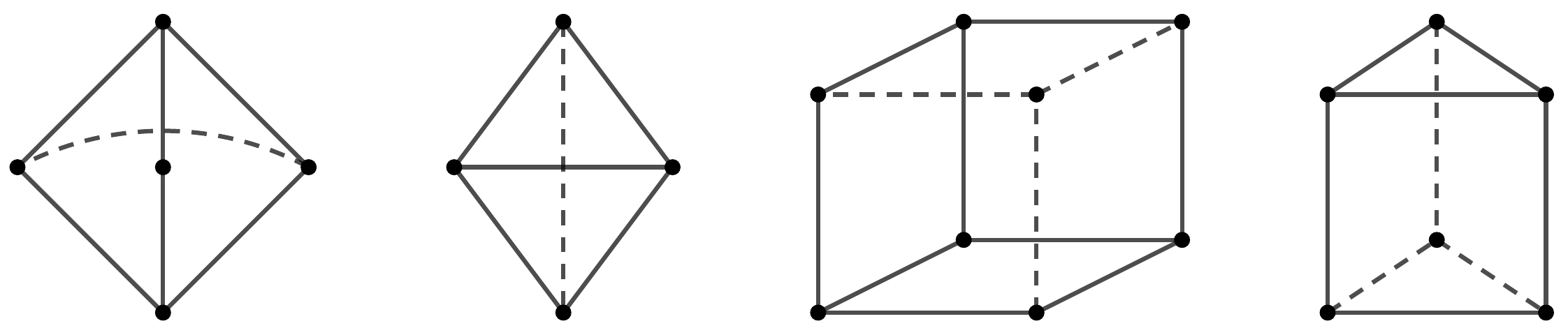}
\caption{Local-to-global characterisation of quasi-median graphs.}
\label{LocalGlobal}
\end{center}
\end{figure}

\begin{proof}
Assume that $X$ satisfies all the conditions above. Because $X$ does not contain induced copy of $K_{2,3}$, $X^{\square\triangle}$ is a \emph{prism complex}, i.e.\ a cellular complex whose cells are \emph{prisms}, namely products of simplices, such that the intersection between any two prisms is either empty or a prism of smaller dimension. Then, a direct application of \cite[Theorem~1]{MR3062742} shows that $X$ is weakly modular. Thus, $X$ is a weakly modular graph without induced copies of $K_{2,3}$ or $K_4^-$, i.e.\ $X$ is a quasi-median graph.

\medskip \noindent
Conversely, assume that $X$ is a quasi-median graph. We know from \cite[Theorem~2.120]{QM} that the prism-completion of $X$, namely the prism complex obtained from $X$ by filling its prisms with products of simplices, is contractible (in fact, can be endowed with a CAT(0) metric). Consequently, its two-skeleton, namely the square-triangle-completion $X^{\square\triangle}$, is simply connected. We also know that $X$ does not contain $K_{2,3}$ nor $K_4^-$ as an induced subgraph, just by definition of quasi-median graphs. It remains to verify that the $3$-cube and $3$-prism conditions are satisfied.

\medskip \noindent
An induced copy of the house graph is necessarily isometrically embedded. So we can apply the triangle condition and conclude that our house graph is contained in a $3$-prism $K_2 \times K_3$.

\medskip \noindent
An induced copy of $Q_3^-$ is also automatically isometrically embedded in a quasi-median graph, because the three hyperplanes that cross it must be pairwise distinct. Indeed, according to \cite[Lemma~2.25]{QM}, any two distinct cliques that intersect must belong to distinct hyperplanes. Then, we can apply the quadrangle condition and conclude that our $Q_3^-$ is contained in a $3$-cube $Q_3$.  
\end{proof}

\paragraph{Hyperplane-collapses.} We conclude this section with some preliminaries on \emph{hyperplane-collapses} of median graphs. Recall that, given a median graph $X$ and a collection of hyperplanes $\mathcal{J}$, the \emph{hyperplane-collapse} $X\backslash\backslash \mathcal{J}$ is the median graph obtained by cubulating the wallspace $(X, \text{hyperplanes not in } \mathcal{J})$. Alternatively, $X\backslash\backslash \mathcal{J}$ can be described as the graph obtained from $X$ by collapsing all the edges that belong to hyperplanes in $\mathcal{J}$. 

\medskip \noindent
First, we record how distances can be computed in hyperplane-collapses:

\begin{lemma}[{\cite[Proposition~4.8]{MR2059193}}]\label{lem:DistCollapse}
Let $X$ be a median graph and $\mathcal{J}$ a collection of hyperplanes. Denote by $\pi$ the canonical projection $X \twoheadrightarrow X \backslash \backslash \mathcal{J}$. For all vertices $x,y \in V(X)$,
$$d(\pi(x),\pi(y)) = \# \{ \text{hyperplanes not in $\mathcal{J}$ separating $x$ and $y$} \}.$$
\end{lemma}

\noindent
As an easy consequence of this lemma, it can be shown that:

\begin{cor}\label{cor:DistancesCollapses}
Let $X$ be a median graph and $\bigsqcup_{i \in I} \mathcal{H}_i$ a partition of the set of its hyperplanes. For every $i \in I$, let $\pi_i$ denote the canonical projection $X \twoheadrightarrow X\backslash\backslash \mathcal{H}_i$. Then,
$$d(x,y)= \sum\limits_{i \in I} d(\pi_i(x),\pi_i(y)) \text{ for all } x,y \in V(X).$$
\end{cor}

\begin{proof}
For all vertices $x,y \in V(X)$, we have
$$\begin{array}{lcl} d(x,y) & = & \text{number of hyperplanes separating $x$ and $y$} \\ \\ & =& \displaystyle \sum\limits_{i \in I} \text{ number of hyperplanes in $\mathcal{H}_i$ separating $x$ and $y$} \\ \\ & = & \displaystyle \sum\limits_{i \in I} d(\pi_i(x),\pi_i(y)), \end{array}$$
where the last equality is justified by Corollary~\ref{cor:DistancesCollapses}.
\end{proof}

\noindent
Next, we observe that canonical projections to hyperplane-collapses have convex preimages:

\begin{lemma}\label{lem:PreimageConvex}
Let $X$ be a median graph and $\mathcal{J}$ a non-empty collection of hyperplanes. Denote by $\pi$ the canonical projection $X \twoheadrightarrow X \backslash\backslash \mathcal{J}$. For every vertex $v$ of $X\backslash\backslash \mathcal{J}$, the preimage $\pi^{-1}(v)$ induces a proper convex subgraph in $X$. 
\end{lemma}

\begin{proof}
We have
$$\begin{array}{lcl} \pi^{-1}(v) & = & \{ x \in V(X) \mid d(\pi(x),\pi(v))=0 \} \\ \\ & = & \{ x \in V(X) \mid \text{$x$ and $v$ not separated by any hyperplane from $\mathcal{J}$} \} \\ \\ & = & \bigcap \{\text{halfspaces containing $v$ delimited by the hyperplanes not in $\mathcal{J}$}\}, \end{array}$$
where the second equality is justified by Lemma~\ref{lem:DistCollapse}. As an intersection of halfspaces, $\pi^{-1}(v)$ must be convex. Moreover, it is clear from our description that $\pi^{-1}(v)$ is a proper subset of $V(X)$ if and only if $\mathcal{J}$ is non-empty. 
\end{proof}

\subsection{Relative numbers of ends}\label{section:Ends}

\noindent
In this section, we record some elementary definitions and properties related to the two notions of relative numbers of ends that we use in this article. We start by recalling the standard definition of the number of ends of a graph.

\begin{definition}
Let $X$ be a graph. Its \emph{number of ends} $e(X)$ is the maximal number, possibly infinite, of unbounded components in the complement of a ball.
\end{definition}

\noindent
The first notion of relative number of ends that we are interested in is the following:

\begin{definition}
Let $G$ be a group, $S \subset G$ a subset, and $H \leq G$ a subgroup. The \emph{Schreier graph} $\mathrm{Sch}(G,H,S)$ is the graph whose vertices are the cosets $Hg$ for $g \in G$ and whose edges connect any two cosets $Hg$ and $Hgs$ for $g \in G$ and $s \in S$. 
\end{definition}

\noindent
In the following, our group $G$ will be always finitely generated and we will choose for $S$ a finite generating subset. In order to shorten the notation, we will often write $\mathrm{Sch}(G,H)$ instead of $\mathrm{Sch}(G,H,S)$ when the specific choice of the finite generating set $S$ does not matter for our purpose.

\begin{definition}
Let $G$ be a finitely generated group and $H \leq G$ a subgroup. The \emph{number of ends of $G$ relative to $H$}, denoted by $e(G,H)$, is the number of ends of a Schreier graph $\mathrm{Sch}(G,H)$. 
\end{definition}

\noindent
Here, it is understood that a Schreier graph $\mathrm{Sch}(G,H,S)$ always has the same number of ends regardless of the finite generating set $S$ we choose. 

\noindent
The second notion of relative number of ends that we are interested in is more geometric in nature, and can be defined for arbitrary graphs:

\begin{definition}
Let $X$ be a graph and $Y \leq X$ a subgraph. A connected component of $X \backslash Y$ is \emph{deep} if it contains vertices arbitrarily far away from $Y$. 
\end{definition}

\begin{definition}
Let $X$ be a graph and $Y \leq X$ a subgraph. The \emph{number of coends} $\tilde{e}(X,Y)$ is
\begin{itemize}
	\item the maximal number of deep components of $X\backslash Y^{+L}$ for $L \geq 0$, if this quantity is finite;
	\item $\omega$ if $X \backslash Y^{+L}$ can have an arbitrarily large number of deep components, but always finite, for $L \geq 0$;
	\item $\omega+1$ if there exists $L \geq 0$ such that $X\backslash Y^{+L}$ has infinitely many deep components. 
\end{itemize}
\end{definition}

\noindent
One easily verify that, given a finitely generated group $G$ and a subgroup $H \leq G$, the number of coends $\tilde{e}(G,H)$ is well-defined, i.e.\ it does not depend on the spefic finite generating set we use to drawy the Cayley graph of $G$.

\paragraph{Coarse separation.} As a first preliminary result, we characterise the number of ends in a way similar to the definition of the number of coends. 

\begin{prop}\label{prop:CoarseSepCodim}
Let $G$ be a finitely generated group, $H \leq G$ a subgroup, and $p \in \mathbb{N}$. The inequality $e(G,H) \geq p$ holds if and only if there exists some $L \geq 0$ such that $G \backslash H^{+L}$ has $\geq p$ $H$-orbits of deep connected components.
\end{prop}

\begin{proof}
Let $\pi : G \twoheadrightarrow \mathrm{Sch}(G,H)$ denote the canonical projection $g \mapsto Hg$. We start by noticing that:

\begin{claim}\label{claim:DistToH}
For every $g \in G$, $d(g,H) = d(\pi(g),H)$.
\end{claim}

\noindent
If $g,gs_1, \ldots, gs_1\cdots s_n$ is a path in $G$ connecting $g$ to $H$, then $Hg,Hgs_1, \ldots, Hgs_1 \cdots s_n=H$ defines a path in $\mathrm{Sch}(G,H)$ connecting $Hg=\pi(g)$ to $H$, hence $d(\pi(g),H) \leq d(g,H)$. Conversely, if $Hg, Hgr_1, Hgr_1 \cdots r_m=H$ is a path in $\mathrm{Sch}(G,H)$ connecting $Hg=\pi(g)$ to $H$, then $g,gr_1, \ldots, gr_1 \cdots r_m \in H$ defines a path in $G$ connecting $g$ to $H$, hence $d(g,H) \leq d(\pi(g),H)$. This proves Claim~\ref{claim:DistToH}. 

\medskip \noindent
Notice that $\pi$ sends edges to edges or vertices, so it sends a path to a path. Since it also preserves the distance to $H$ according to Claim~\ref{claim:DistToH}, it follows that, given an $L \geq 0$, $\pi$ sends (deep) connected components of $G\backslash H^{+L}$ to (unbounded) connected components of $\mathrm{Sch}(G,H) \backslash B(H,L)$. Clearly, two components of $G\backslash H^{+L}$ are sent to the same component of $\mathrm{Sch}(G,H)$ precisely when they belong to the same $H$-orbit. Consequently, the number of unbounded connected components of $\mathrm{Sch}(G,H) \backslash B(H,L)$ coincides with the number of deep connected components of $G\backslash H^{+L}$. This assertion being true for every $L \geq 0$, our lemma follows immediately.
\end{proof}

\paragraph{Almost invariant subsets.} It is well-known that codimension-one subgroups, i.e.\ subgroups for which the relative number of ends is $\geq 2$, are characterised by \emph{almost invariant subsets}. Below, we generalise this approach in order to characterise the number of relative ends and the number of coends.

\begin{prop}\label{prop:AlmostInCoarseSep}
Let $G$ be a finitely generated group, $H \leq G$ a subgroup, and $p \geq 0$ an integer. The equality $\tilde{e}(G,H) \geq p$ holds if and only if there exist $A_1, \ldots, A_p \subset G$ such that:
\begin{itemize}
	\item $A_1, \ldots, A_p$ are parwise disjoint;
	\item $A_1, \ldots, A_p$ are all $H$-infinite;
	\item for all $1 \leq i \leq p$ and $g \in G$, $A_i \cap A_i^c g$ is $H$-finite.
\end{itemize}
\end{prop}

\begin{prop}\label{prop:AlmostInvCodim}
Let $G$ be a finitely generated group, $H \leq G$ a subgroup, and $p \geq 0$ an integer. The equality $e(G,H) \geq p$ holds if and only if there exist $A_1, \ldots, A_p \subset G$ such that:
\begin{itemize}
	\item $A_1, \ldots, A_p$ are all $H$-invariant;
	\item $A_1, \ldots, A_p$ are parwise disjoint;
	\item $A_1, \ldots, A_p$ are all $H$-infinite;
	\item for all $1 \leq i \leq p$ and $g \in G$, $A_i \cap A_i^c g$ is $H$-finite.
\end{itemize}
\end{prop}

\noindent
The last conditions in these two propositions are clarified by our next lemma, which provides a more geometric but equivalent condition. Recall that, given a graph $X$ and set $S \subset V(X)$ of vertices, the \emph{boundary} $\partial S$ is the set of vertices not in $S$ but adjacent to at least one vertex in $S$. 

\begin{lemma}
Let $G$ be a finitely generated group, $H \leq G$ a subgroup, and $A \subset G$ a subset. The following assertions are equivalent:
\begin{itemize}
	\item $\partial A$ is $H$-finite;
	\item for every $g \in G$, $A \cap A^cg$ is $H$-finite.
\end{itemize}
\end{lemma}

\begin{proof}
In this proof, we fix a finite generating set $S$ of $G$ and think about $G$ metrically as its Cayley graph associated to $S$. Notice that 
$$\begin{array}{lcl} \partial A & = & \displaystyle \{ x \in G \mid x \notin A, x \in A^{+1} \} = \bigcup\limits_{s \in S} \{x \in G \mid x \notin A, xs \in A \} \\ \\ & = & \displaystyle \bigcup\limits_{s \in S} \{ x \in G \mid xs^{-1} \notin A, x \in A\} = \bigcup\limits_{s \in S} A \cap A^cs. \end{array}$$
If we know that $A \cap A^cg$ is finite for every $g \in G$, then it follows that $\partial A$ can be written as a union of finitely many $H$-finite subsets, proving that $\partial A$ is $H$-finite, as desired. 

\medskip \noindent
Then, notice that, given a $g \in G$, $A \cap A^cg$ is contained in the $\|g\|_S$-neighbourhood of $\partial A$. Indeed, if $x \in A \cap A^cg$, then a geodesic connecting $x \in A$ to $xg^{-1} \in A^c$, which has length $\|g\|_S$, must intersect $\partial A$, proving that $x$ lies at distance $\leq \|g\|_S$ from $\partial A$. Since a neighbourhood of an $H$-finite subset is stille $H$-finite, we conclude that, if $\partial A$ is finite then $A \cap A^cg$ is $H$-finite for every $g \in G$.
\end{proof}

\begin{proof}[Proof of Propositions~\ref{prop:AlmostInCoarseSep} and~\ref{prop:AlmostInvCodim}.]
Fix an $L \geq 0$ and let $C_1, \ldots, C_n$ denote the deep connected components of $G\backslash H^{+L}$. Clearly,
\begin{itemize}
	\item $C_1,\ldots, C_n$ are pairwise disjoint;
	\item each $C_i$ is $H$-infinite (since $C_i$ is deep);
	\item each $\partial C_i$ is $H$-finite (since contained in $H^{+L}$).
\end{itemize}
Let $C_{i_1}, \ldots, C_{i_m}$ be representatives of our components modulo the action of $H$. For every $1 \leq k \leq m$, set $B_k:= \bigcup_{h \in H} hC_{i_k}$. Then
\begin{itemize}
	\item each $B_k$ is $H$-invariant;
	\item $B_1, \ldots, B_m$ are pairwise disjoint;
	\item each $B_k$ is $H$-infinite (since they contain $H$-infinite subsets by construction);
	\item each $\partial B_k$ is $H$-finite since $\partial B_k \subset \bigcup_{h \in H} h \partial C_{i_k} \subset H^{+L}$.
\end{itemize}
This proves half of our propositions. Conversely, assume that there exist pairwise disjoint $H$-infinite subsets $A_1, \ldots, A_p \subset G$ with $H$-finite boundaries. 

\medskip \noindent
Fix an $L \geq 0$ such that $\partial A_1 \cup \cdots \cup \partial A_p \subset H^{+L}$. Then, fix an $R \geq 0$ such that every component of $G\backslash H^{+L}$ that is not deep is contained in $H^{+R}$. For every $1 \leq i \leq p$, let $a_i \in A_i$ be a point satisfying $d(a_i,H)>R$. (Such a point exists because $A_i$ is $H$-infinite.) Let $C_i$ denote the connected component of $G\backslash H^{+L}$ containing $a_i$. 

\medskip \noindent
Notice that $C_i$ is deep by construction; and that $C_i \subset A_i$, since otherwise $C_i$ would have to intersect $\partial A_i$, which is impossible since $\partial A_i \subset H^{+L}$. Thus, we have found deep connected components $C_1, \ldots, C_p$, which are pairwise distinct since the $A_i$ are pairwise disjoint. It follows that $\tilde{e}(G,H) \geq p$. 

\medskip \noindent
If we know moreover that the $A_i$ are $H$-invariant, then $H \cdot C_i \subset A_i$ for every $1 \leq i \leq p$, which implies that $C_1, \ldots, C_p$ lie in pairwise distinct $H$-orbits since the $A_i$ are pairwise disjoint. Then, it follows from Proposition~\ref{prop:CoarseSepCodim} that $e(G,H) \geq p$. 
\end{proof}

\section{From quasi-median to median}\label{section:QMtoM}

\noindent
In this section, we show that every quasi-median graph is closely related to a median graph canonically associated to it. Such a connection should not be surprising, since every every quasi-median graph can be turned into a median graph by cubulating the wallspace given by the walls $\{\text{sector}, \text{cosector}\}$. (See \cite[Proposition~4.16]{QM} and the related discussion for details.) In fact, our construction below turns out to be equivalent to this cubulation, but the persective we propose here, more explicit, will be easier to work with.

\medskip \noindent
Our median models of quasi-median graphs, which we call \emph{graph of prisms}, are defined and studied in Section~\ref{section:GraphPrisms}. In order to prove that graph of prisms are median, it will be convenient to look at them as subgraphs in other median graphs associated to quasi-median graphs, namely \emph{graphs of polytopes}. Section~\ref{section:GraphPoly} is dedicated to these graphs.

\subsection{Graph of polytopes}\label{section:GraphPoly}

\noindent
This section is dedicated to \emph{graphs of polytopes} associated to quasi-median graphs. Vertices of such graphs are specific gated subgraphs in the corresponding quasi-median graphs, namely:

\begin{definition}
Let $X$ be a quasi-median graph. A non-empty subgraph of $X$ is a \emph{polytope} if it is the gated hull of finitely many vertices. 
\end{definition}

\noindent
Polytopes can be defined with respect to any abstract convexity \cite{MR1234493}. Here, we use the convexity induced by gated subgraphs, which is often more relevant, in the context of quasi-median graphs, than the geodesic convexity. We emphasize, however, that we do not allow our polytopes to be empty, contrary to \cite[\S 1]{MR1234493}. 

\medskip \noindent
Examples of polytopes includes singletons, cliques, and prisms. In particular, polytopes may not be finite. The structure of polytopes is clarified by the following observation:

\begin{lemma}[{\cite[Corollaries~2.71 and~2.81]{QM}}]\label{lem:WhenPolytope}
Let $X$ be a quasi-median graph and $Y \leq X$ a gated subgraph. The following assertions are equivalent: 
\begin{itemize}
	\item $Y$ is a polytope;
	\item $Y$ has only finitely many hyperplanes;
	\item $Y$ is a union of finitely many prisms.
\end{itemize}
\end{lemma}

\noindent
We encode polytopes of a quasi-median graphs into a single graph as follows:

\begin{definition}
Let $X$ be a quasi-median graph. The \emph{graph of polytopes} $\mathrm{Poly}(X)$ is the graph 
\begin{itemize}
	\item whose vertices are the polytopes of $X$;
	\item and whose edges connect any two polytopes $P \leq Q$ whenever there is no polytope $R$ satisfying $P<R<Q$.
\end{itemize}
\end{definition}

\noindent
In other words, the graph of prisms is the covering graph of the poset $(\{\text{polytopes}\}, \subseteq)$. Recall that, given a poset $(E,\leq)$, the \emph{covering relation} $\lessdot$ is defined by: $a \lessdot b$ if $a<b$ and there is no $c \in E$ satisfying $a<c<b$. Then, the \emph{covering graph} of $(E,\leq)$ is the graph whose vertex-set is $E$ and whose edges connect any two $a,b \in E$ satisfying $a\lessdot b$. 

\medskip \noindent
Our next lemma describes more precisely the covering relation $\lessdot$ between polytopes of quasi-median graphs. For convenience, it uses the following notation: given a quasi-median graph $X$ and subgraphs $A,B \leq X$,
\begin{itemize}
	\item $\mathcal{H}(A)$ denotes the set of the hyperplanes that separate at least two vertices of $A$;
	\item $\mathfrak{h}(A)$ denotes the cardinality of $\mathcal{H}(A)$;
	\item $\mathcal{H}(A | B)$ denotes the set of the hyperplanes that separates $A$ and $B$ (i.e.\ that contain $A$ and $B$ into two distinct sectors); 
	\item $\mathfrak{h}(A|B)$ denotes the cardinality of $\mathcal{H}(A|B)$. 
\end{itemize}
We are now ready to state and prove our characterisation.

\begin{lemma}\label{lem:Lessdot}
Let $X$ be a quasi-median graph and $A,B \leq X$ two polytopes. The following assertions are equivalent:
\begin{itemize}
	\item[(i)] $A \lessdot B$;
	\item[(ii)] $A<B$ and $B= \mathrm{GatedHull}(A \cup \{x\})$ for every $x \in \partial_BA$, where $$\partial_BA:= \{ p \in V(B) \mid p \text{ not in $A$ but adjacent to some vertex in } A\}.$$ 
	\item[(iii)] $A \leq B$ and $\mathfrak{h}(B)= \mathfrak{h}(A)+1$;
\end{itemize}
\end{lemma}

\begin{proof}
We start by proving a pair of elementary but useful observations:

\begin{claim}\label{claim:AddingVertexGated}
Let $Y \leq X$ be a gated subgraph and $p \in \partial_XY$. If $J$ denotes the hyperplane separating $p$ from a neighbour in $Y$, then
$$\mathcal{H}(\mathrm{GatedHull}(Y \cup \{p\})) = \mathcal{H}(Y) \sqcup \{J\}.$$
\end{claim}

\noindent
We know from \cite[Proposition~2.68]{QM} that a hyperplane crosses the gated hull of $Y \cup \{p\}$ if and only if it separates at least two vertices in $Y \cup \{p\}$. The hyperplanes separating two vertices in $Y$ are the hyperplanes of $Y$, and there is a unique hyperplane separating $p$ from some vertex of $Y$, namely $J$. 

\begin{claim}\label{claim:InclusionHyp}
Let $Y,Z \leq X$ be two gated subgraphs. If $Y \leq Z$ and $\mathcal{H}(Y)= \mathcal{H}(Z)$, then $Y=Z$. 
\end{claim}

\noindent
If $Y \leq Z$ but $Y \neq Z$, then there exists a vertex $z \in V(Z \backslash Y)$. As a consequence of Lemma~\ref{lem:SepProj}, there exists a hyperplane $J$ separating $z$ from $Y$. Of course, $J \in \mathcal{H}(Z) \backslash \mathcal{H}(Y)$, so we cannot have $\mathcal{H}(Y) = \mathcal{H}(Z)$, concluding the proof of Claim~\ref{claim:InclusionHyp}. 

\medskip \noindent
Assume that $A \lessdot B$. Since $A<B$, $\partial_BA$ contains at least one vertex, say $x$. Then, $A< \mathrm{GatedHull}(A \cup \{x\}) \leq B$. So we must have $B= \mathrm{GatedHull}(A \cup \{x\})$. Thus, (i) implies (ii). We also know from Claim~\ref{claim:AddingVertexGated} that (ii) implies (iii). 

\medskip \noindent
Now, assume that (iii) holds. Let $C \leq X$ be a polytope satisfying $A \leq C \leq B$. Of course, $\mathcal{H}(A) \subset \mathcal{H}(C) \subset \mathcal{H}(B)$. Since $\mathfrak{h}(B)= \mathfrak{H}(A)+1$, necessarily $\mathcal{H}(C)=\mathcal{H}(A)$ or $\mathcal{H}(B)$. Then, it follows from Claim~\ref{claim:InclusionHyp} that $C=A$ or $B$. Hence $A \lessdot B$. This proves that (iii) implies (i). 
\end{proof}

\noindent
The rest of the section is dedicated to the proof of the following statement:

\begin{thm}\label{thm:GraphPoly}
The graph of polytopes of a quasi-median graph is median.
\end{thm}

\noindent
For median graphs, a proof can be found in \cite[Section~9.1]{QM}. (See also \cite{MR4449680} for a similar construction in arbitrary median metric spaces.) An alternative argument can also be extracted from the earlier work \cite{MR529794}, based on the characterisation of median graphs as covering graphs of discrete median semilattices. 

\medskip \noindent
In order to prove Theorem~\ref{thm:GraphPoly}, we start by computing distances in graphs of polytopes. 

\begin{prop}\label{prop:DistPoly}
Let $X$ be a quasi-median graph and $A,B \leq X$ two polytopes. Then,
$$\begin{array}{lcl} d_{\mathrm{Poly}(X)}(A,B) & = & \mathfrak{h}(A\backslash B) + 2 \mathfrak{h}(A|B) +  \mathfrak{h}(B \backslash A) \\ \\ & = & 2 \mathfrak{h}(A \cup B) - \mathfrak{h}(A) - \mathfrak{h}(B). \end{array}$$
\end{prop}

\noindent
Our proof of the proposition will be based on the following particular case:

\begin{lemma}\label{lem:DistPolyNested}
Let $X$ be a quasi-median graph and $A, B$ two polytopes. If $A \leq B$, then
$$d_{\mathrm{Poly}(X)}(A,B) = \mathfrak{h}(B \backslash A).$$
\end{lemma}

\begin{proof}
It follows from Lemma~\ref{lem:Lessdot} that, along a path in $\mathrm{Poly}(X)$, the number of hyperplanes of a polytope increase or decrease by one at each step. Therefore, we must have $d_{\mathrm{Poly}(X)}(A,B) \geq \mathfrak{h}(B \backslash A)$. Conversely, we can construct a path of length $\mathfrak{h}(B\backslash A)$ in $\mathrm{Poly}(X)$ connecting $A$ to $B$ as follows. If $A=B$, there is nothing to prove, so assume that $A<B$. Consequently, the boundary $\partial_BA$ contains at least one vertex, say $p$. According to Lemma~\ref{lem:Lessdot} and Claim~\ref{claim:AddingVertexGated}, the gated hull of $A \cup \{p\}$ is adjacent to $A$ in $\mathrm{Poly}(X)$ and its set of hyperplanes is $\mathcal{H}(A) \sqcup \{J\}$, where $J$ is the unique hyperplane separating $p$ from $A$, which is also a hyperplane of $B$. In particular, $\mathrm{GatedHull}(A \cup \{p\}) \leq B$ has one more hyperplane in common with $B$ than $A$. Thus, after $\mathfrak{h}(B\backslash A)$ iterations of this construction, we obtain a polytope $A^+ \leq B$ with $\mathcal{H}(A^+)= \mathcal{H}(B)$, which implies that $A^+=B$ according to Claim~\ref{claim:InclusionHyp}. 
\end{proof}

\begin{proof}[Proof of Proposition~\ref{prop:DistPoly}.]
Let $C$ denote the gated hull of $A \cup B$. It follows from Lemma~\ref{lem:DistPolyNested} that
$$d_{\mathrm{Poly}(X)}(A,B) \leq d_{\mathrm{Poly}(X)}(A,C) + d_{\mathrm{Poly}(X)}(C,B) = \mathfrak{h}(C\backslash A) + \mathfrak{h}(C\backslash B).$$
We know from \cite[Proposition~2.68]{QM} that the hyperplanes crossing $C$ are exactly the hyperplanes separating at least two vertices in $A \cup B$. Hence
$$\mathfrak{h}(C\backslash A)= \mathfrak{h}(A|B) + \mathfrak{h}(B\backslash A) \text{ and } \mathfrak{h}(C\backslash B)= \mathfrak{h}(A|B) + \mathfrak{h}(A\backslash B).$$
So far, we have proved that 
$$d_{\mathrm{Poly}(X)}(A,B) \leq \mathfrak{h}(A\backslash B) + 2 \mathfrak{h}(A|B) +  \mathfrak{h}(B \backslash A).$$
The reverse equality is clear. Indeed, we know from Lemma~\ref{lem:Lessdot} that, along a path connecting $A$ to $B$ in $\mathrm{Poly}(X)$, one adds or removes one hyperplane at each step. And, along the path, every hyperplane crossing $A$ but not $B$ has to be removed at some point, every hyperplane crossing $B$ but not $A$ has to be added at some point, and every hyperplane separating $A$ and $B$ has to be added and then remove at some points. This proves the first equality of our proposition.

\medskip \noindent
Then, notice that the expression $2\mathfrak{h}(A \cup B)-\mathfrak{h}(A)-\mathfrak{h}(B)$ counts ($2-1-0=$) once the hyperplanes crossing $A$ but not $B$, ($2-0-1=$) once the hyperplanes crossing $B$ but not $B$, ($2-1-1=$) none the hyperplanes crossing both $A$ and $B$, and ($2-0-0=$) twice the hyperplanes separating $A$ and $B$. Hence the equality
$$\mathfrak{h}(A\backslash B)+2\mathfrak{h}(A|B) +\mathfrak{h}(B\backslash A) = 2 \mathfrak{h}(A \cup B)- \mathfrak{h}(A)- \mathfrak{h}(B),$$
which concludes the proof of our proposition. 
\end{proof}

\noindent
As a consequence of Proposition~\ref{prop:DistPoly}, it is possible to describe geodesics in graphs of polytopes.

\begin{cor}\label{cor:IntervalPoly}
Let $X$ be a quasi-median graph and $A,B,C \leq X$ three polytopes. In $\mathrm{Poly}(X)$, $C$ belongs to a geodesic connecting $A$ to $B$ if and only if the following conditions are satisfied:
\begin{itemize}
	\item every sector containing $C$ also contains $A$ or $B$;
	\item every sector containing $A$ and $B$ also contains $C$;
\end{itemize}
\end{cor}

\begin{proof}
By applying the formula given by Proposition~\ref{prop:DistPoly}, we can compute $d(A,C)+d(C,B)$ and compare its value with $d(A,B)$. This is done on Figure~\ref{Computation}. It follows from this computation that $C$ belongs to a geodesic connecting $A$ and $B$, i.e.\ $d(A,C)+d(C,B)=d(A,B)$, if and only if the three conditions above are satisfied.
\begin{figure}
\begin{center}
\includegraphics[width=\linewidth]{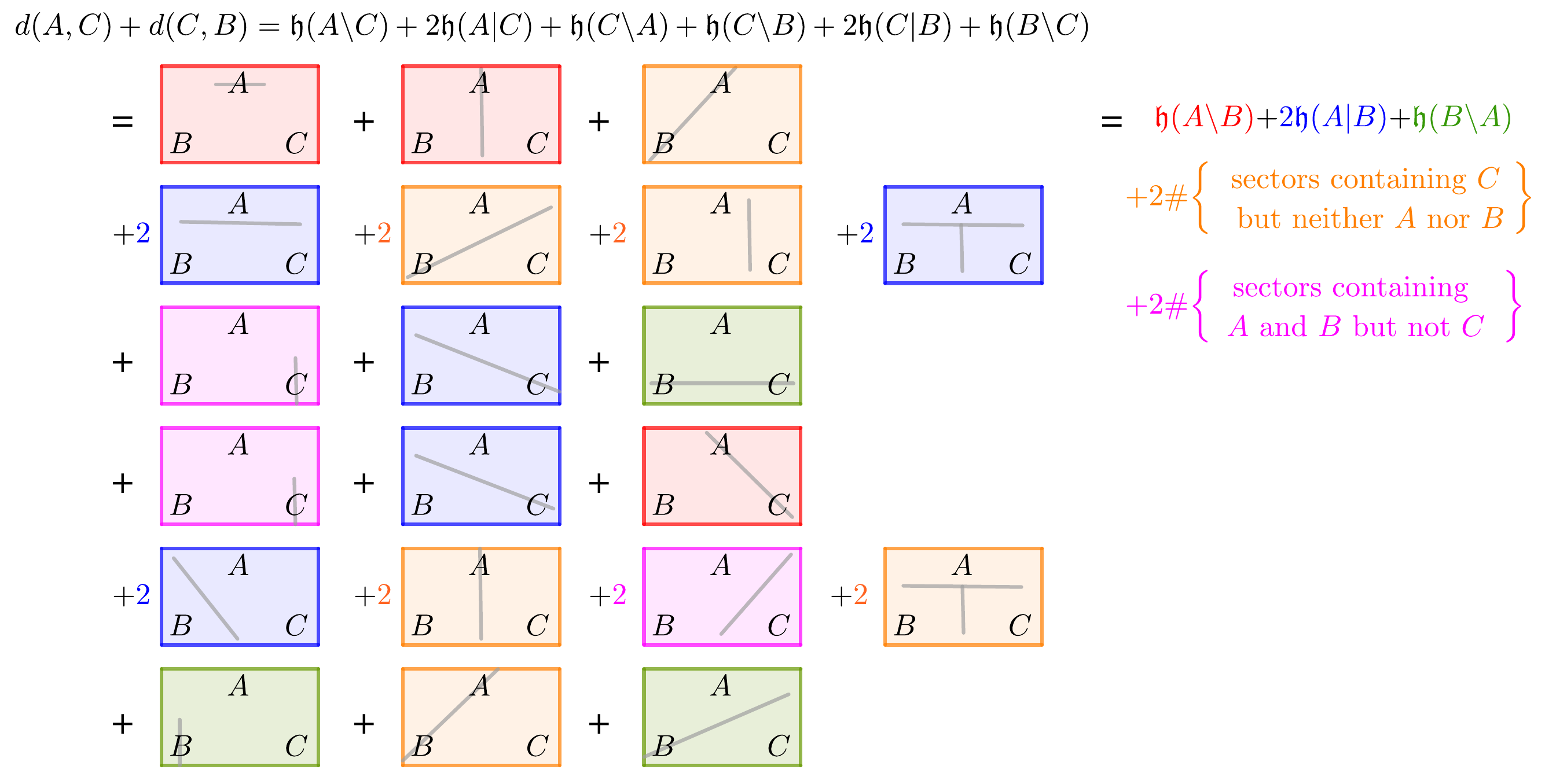}
\caption{Proof of Corollary~\ref{cor:IntervalPoly}.}
\label{Computation}
\end{center}
\end{figure}
\end{proof}

\noindent
We are now ready to prove the main result of this section. 

\begin{proof}[Proof of Theorem~\ref{thm:GraphPoly}.]
In order to show that the graph of polytopes of a given quasi-median graph $X$ is median, our goal is to prove:

\begin{claim}\label{claim:MedianPoly}
Let $A_1,A_2,A_3 \leq X$ be three polytopes. The unique median $M$ of $\{A_1,A_2,A_3\}$ in $\mathrm{Poly}(X)$ is the intersection of all the sectors containing at least two polytopes among $A_1,A_2,A_3$. Moreover, a hyperplane crosses $M$ if and only if each of its sectors contains at most one polytope among $A_1,A_2,A_3$.
\end{claim}

\noindent
The first step is to justify that $M$ is non-empty. Since the gated hull $A^+$ of $A_1 \cup A_2 \cup A_3$ coincides with the intersection of all the sectors containing $A_1 \cup A_2 \cup A_3$ (\cite[Lemma~2.69]{QM}), we can write
\begin{equation}\label{eq:M}
M = A^+ \cap \bigcap \{ \text{sectors of } A^+ \text{ containing at least two polytopes among } A_1,A_2,A_3 \}
\end{equation}
Notice that $A^+$ is a polytope, and consequently, according to Lemma~\ref{lem:WhenPolytope}, has only finitely many hyperplanes. Thus, $M$ is an intersection of finitely many gated subgraphs that pairwise intersect, which implies that $M$ is indeed non-empty. 

\medskip \noindent
Next, let us verify that a hyperplane $J$ of $X$ crosses $M$ if and only if it does not delimit a sector containing two polytopes among $A_1,A_2,A_3$. It is clear that from the definition of $M$ that, if $J$ delimits a sector containing two polytopes among $A_1,A_2,A_3$, then $J$ does not cross $M$. Conversely, assume that $J$ does not cross $M$. In other words, $J$ delimits some sector $S$ containing $M$. It follows from Lemma~\ref{lem:InclusionOne} and from the decomposition (\ref{eq:M}) of $M$ that $S$ contains $A^+$ or a sector of $A^+$ containing at least two polytopes among $A_1,A_2,A_3$. A fortiori, $S$ contains at least two polytopes among $A_1,A_2,A_3$, as desired. This proves the second assertion of Claim~\ref{claim:MedianPoly}. 

\medskip \noindent
It also follows from this description of the hyperplanes crossing $M$, combined with the definition of $M$ and Corollary~\ref{cor:IntervalPoly}, that $M$ is a median of $\{A_1,A_2,A_3\}$ in $\mathrm{Poly}(X)$. It remains to verify that this is the only such median. 

\medskip \noindent
Assume that $M_0 \leq X$ is a median of $\{A_1,A_2,A_3\}$ in $\mathrm{Poly}(X)$. It follows from Corollary~\ref{cor:IntervalPoly} that every sector containing two polytopes among $A_1,A_2,A_3$ must also contain $M_0$, hence $M_0 \leq M$. If $M_0<M$, then it follows from Lemma~\ref{lem:SepProj} that there exists some hyperplane $J$ crossing $M$ but not $M_0$. Let $S$ denote the sector delimited by $J$ that contains $M_0$. Since $M_0$ belongs to a geodesic connecting $A_1$ to $A_2$ in $\mathrm{Poly}(X)$, it follows from Corollary~\ref{cor:IntervalPoly} that $S$ contains $A_1$ or $A_2$, say $A_1$ up to reindexing our polytopes. Similarly, because $M_0$ belongs to a geodesic connecting $A_2$ to $A_3$ in $\mathrm{Poly}(X)$, $S$ must contain $A_2$ or $A_3$. Thus, $S$ contains at least two polytopes among $A_1,A_2,A_3$, contradicting our description of the hyperplanes crossing $M$. In other words, we must have $M_0=M$.
\end{proof}

\subsection{Graph of prisms}\label{section:GraphPrisms}

\noindent
In this section, we define and study \emph{graphs of prisms} of quasi-median graphs, which will play an important role in the proof of Theorem~\ref{thm:BigIntro}.

\begin{definition}
Let $X$ be a quasi-median graph. Its \emph{graph of prisms} $\mathbb{P}(X)$ is the graph
\begin{itemize}
	\item whose vertices are the prisms of $X$;
	\item and whose edges connect any two prisms $P \leq Q$ whenever there is no prism $R$ satisfying $P < R < Q$
\end{itemize}
\end{definition}

\noindent
In other words, the graph of prisms is the covering graph of the poset $(\{\text{prisms}\}, \subseteq)$. Notice that, since prisms are polytopes, graphs of prisms are naturally subgraphs of graphs of polytopes. 

\medskip \noindent
Our first goal in this section is to prove that:

\begin{thm}\label{thm:PrismMedian}
The graph of prisms of a quasi-median graph is a median graph.
\end{thm}

\noindent
Our proof of the theorem will use the fact that graphs of polytopes are median combined with the following observation:

\begin{lemma}\label{lem:MedianClosed}
Let $X$ be a median graph and $Y \leq X$ a subgraph. If $Y$ is connected and median-closed (i.e.\ the median point of $\{a,b,c\}$ belongs to $Y$ for all $a,b,c \in V(Y)$), then $Y$ is isometrically embedded in $X$. In particular, $Y$ is a median graph.
\end{lemma}

\noindent
In fact, it can be proved that, in a median graph, a subgraph is connected and median-closed if and only if it is a retract. See \cite[Chapter~2.4]{Book} for details. The weaker assertion provided by Lemma~\ref{lem:MedianClosed} will be sufficient for our purpose. 

\begin{proof}[Proof of Lemma~\ref{lem:MedianClosed}.]
Let $a,b \in V(Y)$ be two vertices. Because $Y$ is connected, there exists a path
$$x_1:=a, x_2, \ldots, x_{n-1}, x_n:=b$$
in $Y$. If this is a geodesic in $X$, then there is nothing to prove, so we assume that our path is not a geodesic in $X$. Fix a subpath $x_i, \ldots, x_j$ ($1 \leq i<j \leq n$) of minimal length that is not a geodesic. This amounts to saying that the edge $\{x_i,x_{i+1}\}$ and $\{x_{j-1},x_j\}$ belong to the same hyperplane, say $J$, and that $x_{i+1},\ldots, x_{j-1}$ is a geodesic. Since fibres of hyperplanes are convex, the configuration is the following:
\begin{center}
\includegraphics[width=0.6\linewidth]{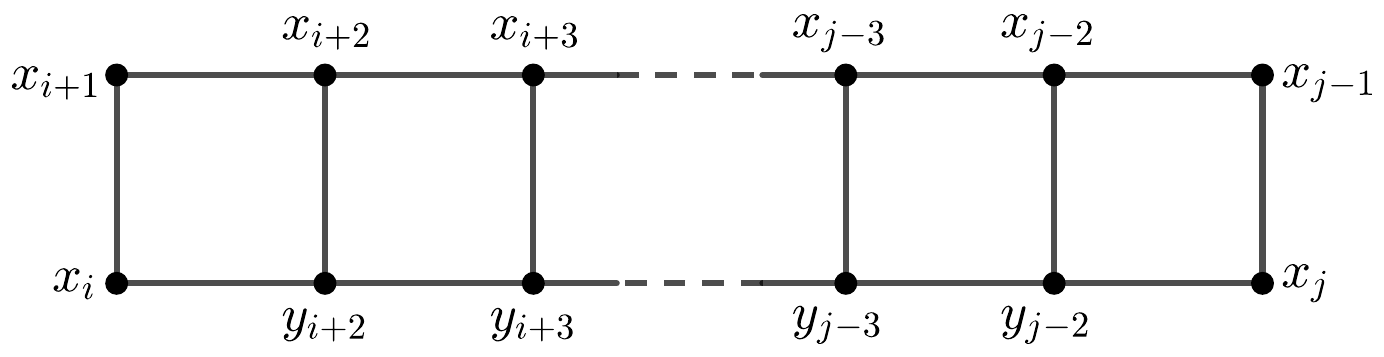}
\end{center}
Notice that, for every $i+2 \leq k \leq j-2$, $y_k$ is the median of $\{x_i,x_k,x_j\}$, and consequently belongs to $Y$. Thus, we can shorten our path in $Y$ by replacing $x_i, x_{i+1},\ldots,x_{j-1}, x_j$ with $x_i,y_{i+2},\ldots, y_{j-2},x_j$. After finitely many iterations, we find a geodesic of $X$ contained in $Y$ that connects $a$ to $b$. 
\end{proof}

\begin{proof}[Proof of Theorem~\ref{thm:PrismMedian}.]
Let $X$ be a quasi-median graph. In order to prove that the graph of prisms $\mathbb{P}(X)$ is median, it suffices to show that $\mathbb{P}(X)$ is a connected median-closed subgraph of the graph of polytopes $\mathrm{Poly}(X)$, which is median according to Theorem~\ref{thm:GraphPoly}. Indeed, the desired conclusion then follows from Lemma~\ref{lem:MedianClosed}. 

\medskip \noindent
Let $P \leq X$ be a prism, which we decompose as a product of cliques $C_1 \times \cdots \times C_n$. For every $1 \leq i \leq n$, fix a vertex $o_i \in V(C_i)$. Then
$$C_1 \times C_2 \times \cdots C_n, \{o_1\} \times C_2 \times \cdots \times C_n, \ldots, \{o_1\} \times \{o_2\} \times \cdots \times \{o_n\}$$
defines a path in $\mathbb{P}(X)$ from our prism $P$ to a singleton. Thus, in order to prove that $\mathbb{P}(X)$ is connected, it is sufficient to verify that, for all vertices $a,b \in V(X)$, $\{a\}$ and $\{b\}$ are connected by a path in $\mathbb{P}(X)$. But it is clear that, given a path $x_1:=a,x_2, \ldots, x_{n-1},x_n:=b$ in $X$, 
$$\{x_1\}, \{x_1,x_2\}, \{x_2\}, \{x_2,x_3\}, \{x_3\}, \ldots, \{x_b\}$$
defines a path in $\mathbb{P}(X)$ connecting $\{a\}$ to $\{b\}$. Thus, we have proved that $\mathbb{P}(X)$ is connected.

\medskip \noindent
It remains to verify that $\mathbb{P}(X)$ is median-closed in $\mathrm{Poly}(X)$. So let $P_1,P_2,P_3 \leq X$ be three prisms. We know from Claim~\ref{claim:MedianPoly} that the median $M$ of $\{P_1,P_2,P_3\}$ in $\mathrm{Poly}(X)$ is the intersection of all the sectors containing at least two prisms among $P_1,P_2,P_3$ and that a hyperplane crosses $M$ if and only if each of its sectors contains at most one prism among $P_1,P_2,P_3$. Our goal is to show that $M$ belongs to $\mathrm{P}(X)$, i.e.\ is a prism, which amounts to saying that the hyperplanes crossing $M$ are pairwise transverse (\cite[Lemma~2.74]{QM}). 

\medskip \noindent
So let $J_1$ and $J_2$ be two hyperplanes of $X$ crossing $M$. Assume for contradiction that $J_1$ and $J_2$ are not transverse. Let $J_1^+$ (resp.\ $J_2^+$) denote the sector delimited by $J_1$ that contains $J_2$ (resp.\ $J_1$). We know that $J_1^+$ contains at most one prism among $P_1,P_2,P_3$. Up to reindexing our prisms, say that $P_1$ and $P_2$ are not contained in $J_1^+$. In other words, they are either crossed by $J_1$ or contained in a sector delimited by $J_1$ distinct from $J_1^+$. But, then, $P_1$ and $P_2$ must be contained in $J_2^+$, contradicting our description of the hyperplanes crossing $M$. So $J_1$ and $J_2$ must indeed be transverse. 
\end{proof}

\noindent
Since now we know from Theorem~\ref{thm:PrismMedian} that graphs of prisms are median, it is natural to investigate the structure of their hyperplanes. We start with a couple of elementary observations.

\medskip \noindent
First, given a quasi-median graph $X$, notice that edges of $\mathbb{P}(X)$ are naturally oriented: given two prisms $P$ and $Q$ that are adjacent in $\mathbb{P}(X)$, either $P \lessdot Q$ (in which case we orient $\{P,Q\}$ from $P$ to $Q$) or $Q \lessdot P$ (in which case we orient $\{P,Q\}$ from $Q$ to $P$). 

\medskip \noindent
Next, notice that the edges of $\mathbb{P}(X)$ are naturally labelled by the sectors of $X$. Indeed, given two prisms $P \lessdot Q$, $P$ is codimension-one face of $Q$, i.e.\ $Q$ decomposes as a product $P \times C$ between $P$ and some clique $C$. Then, we label the edge $\{P,Q\}$ with the sector containing $P$ that is delimited by the hyperplane containing $C$. 

\begin{prop}\label{prop:HypPrism}
Let $X$ be a quasi-median graph. The map
$$\left\{ \begin{array}{ccc} E(\mathbb{P}(X)) & \to & \{\text{sectors of } X\} \\ e & \mapsto & \text{label of } e \end{array} \right.$$
induces a bijection $\{\text{hyperplanes of } \mathbb{P}(X)\} \to \{\text{sectors of } X\}$. Moreover, for every sector $S$ of $X$, the halfspaces delimited by the unique hyperplane labelled by $S$ are 
$$\{ P \text{ prism} \mid P \subset S \} \text{ and } \{ P \text{ prism} \mid P \nsubseteq S \}.$$
\end{prop}

\noindent
In order to prove the proposition, the following characterisation of squares in graphs of prisms will be needed:

\begin{lemma}\label{lem:SquarePrism}
Let $X$ be a quasi-median graph. Every induced $4$-cycle in $\mathbb{P}(X)$ is of the form
\begin{center}
\includegraphics[width=0.4\linewidth]{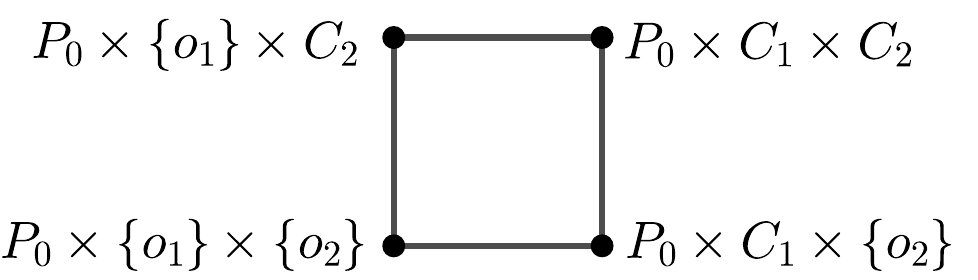}
\end{center}
where $P:=P_0 \times C_1 \times C_2$ is a prism of $X$ that we decompose as a product between a prism $P_0$ and two cliques $C_1,C_2$ and where $o_1 \in V(C_1),o_2 \in V(C_2)$ are two vertices.
\end{lemma}

\begin{proof}
Consider in an induced $4$-cycle in $\mathbb{P}(X)$. 

\medskip \noindent
\begin{minipage}{0.2\linewidth}
\includegraphics[width=0.95\linewidth]{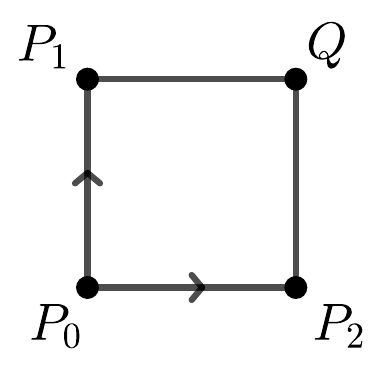}
\end{minipage}
\begin{minipage}{0.78\linewidth}
Fix a vertex represented by a prism $P_0$ of $X$ with minimal cubical dimension. Then, its two neighbours $P_1$ and $P_2$ in our $4$-cycle must satisfy $P_0 \lessdot P_1,P_2$. Let $Q$ denote the prism of $X$ representing the fourth vertex of our $4$-cycle. 
\end{minipage}

\medskip \noindent
First, assume that $Q \lessdot P_1$. 

\medskip \noindent
\begin{minipage}{0.2\linewidth}
\includegraphics[width=0.95\linewidth]{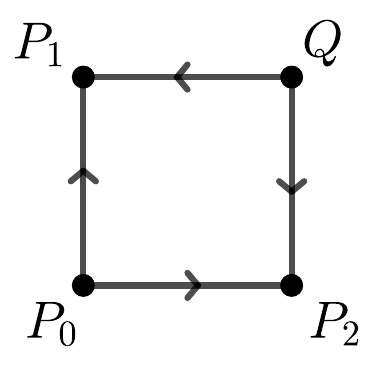}
\end{minipage}
\begin{minipage}{0.78\linewidth}
Then, $Q$ has cubical dimension $\mathrm{dim}_\square(P_1)-1 = \dim_\square(P_0)$. Since $\dim_\square(P_2)= \dim_\square(P_0)+1$, we must also have $Q \lessdot P_2$. For $i=1,2$, let $J_i$ (resp.\ $H_i$) denote the unique hyperplane of $P_i$ that does not cross $P_0$ (resp.\ $Q$). Notice that, because $P_1 \neq P_2$, necessarily $J_1 \neq J_2$. 
\end{minipage}

\medskip \noindent
Notice that
$$(\mathcal{H}(P_0) \sqcup \{J_1\}) \backslash \{H_1\} = \mathcal{H}(Q) = (\mathcal{H}(P_0) \sqcup \{J_1\} ) \backslash \{H_2\},$$
where $\mathcal{H}(\cdot)$ denotes the set of the hyperplanes crossing a given prism. Since $J_1 \notin \mathcal{H}(P_0) \sqcup \{J_2\}$ and $J_2 \notin \mathcal{H}(P_0) \sqcup \{J_1\}$, we must have $H_1=J_1$ and $H_2=J_2$. Thus, given an $i=1,2$, $P_0$ and $Q$ are two codimension-one faces of the prism $P_i$ that are not crossed by the hyperplane $J_i$. This implies that $J_i$ is the unique hyperplane separating $P_0$ and $Q$. Hence $J_1=J_2$, a contradiction. 

\medskip \noindent
Therefore, we must have $P_1 \lessdot Q$. Then,
$$\dim_\square(Q)= \dim_\square(P_1)+1 = \dim_\square(P_0)+2.$$
Since $\dim_\square(P_2)= \dim_\square(P_0)+1$, we must also have $P_2 \lessdot Q$. Thus, $Q$ is a prism containing $P_1$ and $P_2$ as two distinct codimension-one faces and containing $P_0$ as a common codimension-one face of $P_1$ and $P_2$. The desired description follows. 
\end{proof}

\begin{proof}[Proof of Proposition~\ref{prop:HypPrism}.]
Let $\Lambda$ denote our map $E(\mathbb{P}(X)) \to \{\text{sectors of } X\}$. First, notice that $\Lambda$ induces a well-defined map 
$$\bar{\Lambda} : \{\text{hyperplanes of } \mathbb{P}(X)\} \to \{\text{sectors of } X\},$$ 
which amounts to saying that any two edges that belong to the same hyperplane have the same label. Indeed, this follows from Lemma~\ref{lem:SquarePrism}, which implies that two opposite edges in an induced $4$-cycle have the same label. We need to verify that $\bar{\Lambda}$ is bijective.

\medskip \noindent
First, let us justify that $\bar{\Lambda}$ is surjective. So let $S$ be an arbitrary sector of $X$. Fix a clique $C$ contained in the hyperplane delimiting $S$ and let $x \in V(C)$ be the unique vertex of $C$ that belongs to $S$. Then $\{\{x\}, C\}$ is an edge of $\mathbb{P}(X)$ labelled by $S$. 

\medskip \noindent
Next, we want to prove that $\bar{\Lambda}$ is injective, i.e.\ any two edges of $\mathbb{P}(X)$ with the same label belong to the same hyperplane. So let $P_1 \lessdot P_1^+$ and $P_2 \lessdot P_2^+$ be two pairs of prisms such that $\{P_1,P_1^+\}$ and $\{P_2, P_2^+\}$ are two edges of $\mathbb{P}(X)$ with the same label, say the sector $S$ of $X$. 

\medskip \noindent
Fix an $i=1,2$. Decompose the prism $P_i^+$ as a product of cliques $C_1 \times \cdots \times C_n$. Up to reindexing the factors, say that $P_i = \{o_1\} \times C_2 \times \cdots \times C_n$ where $o_1 \in V(C_1)$. Fixing arbitrary vertices $o_2 \in V(C_2), \ldots, o_n \in V(C_n)$, we obtain a ladder

\noindent
\begin{center}
\includegraphics[width=0.8\linewidth]{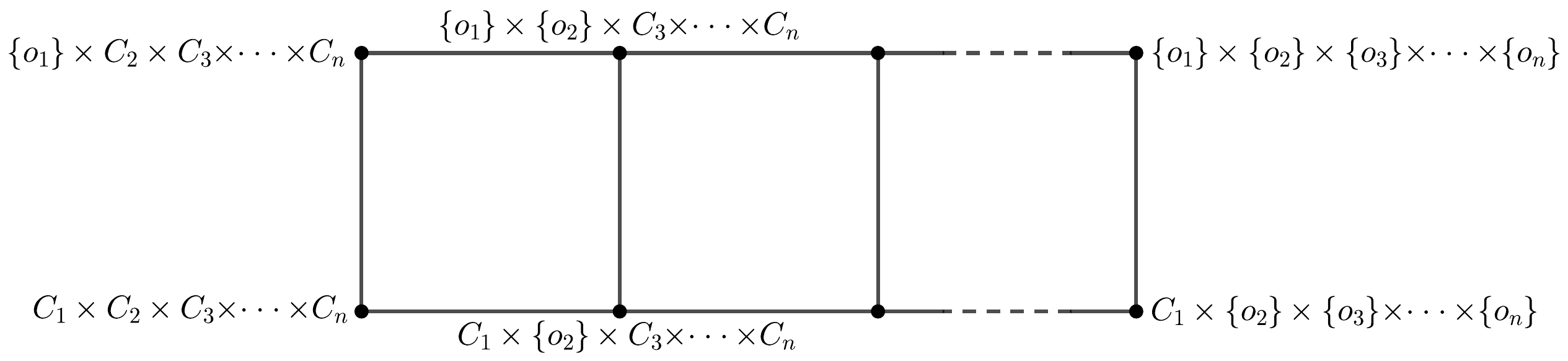}
\end{center}

\noindent
This observation allows us to assume without loss of generality that $P_1,P_2$ are vertices, say $x_1,x_2$, and that $P_1^+,P_2^+$ are cliques, say $C_1,C_2$. Because $C_1$ and $C_2$ belong to the same hyperplane, namely the hyperplane delimiting $S$, we deduce from the product structure of carriers (see \cite[Proposition~2.15]{QM}) that there exists a sequence of cliques 
$$K_1:=C_1, K_2, \ldots, K_{n-1}, K_n:= C_2$$
such that, for every $1 \leq i \leq n-1$, $K_i$ and $K_{i+1}$ are opposite cliques in some prism $P_i$. For every $1 \leq i \leq n$, let $z_i$ denote the unique vertex of $K_i$ that belongs to $S$. Notice that $z_1=x_1$ and $z_n=x_n$. For every $1 \leq i \leq n-1$, $z_i$ and $z_{i+1}$ are adjacent, so they belong to a unique clique $Q_i$. 

\noindent
\begin{center}
\includegraphics[width=0.6\linewidth]{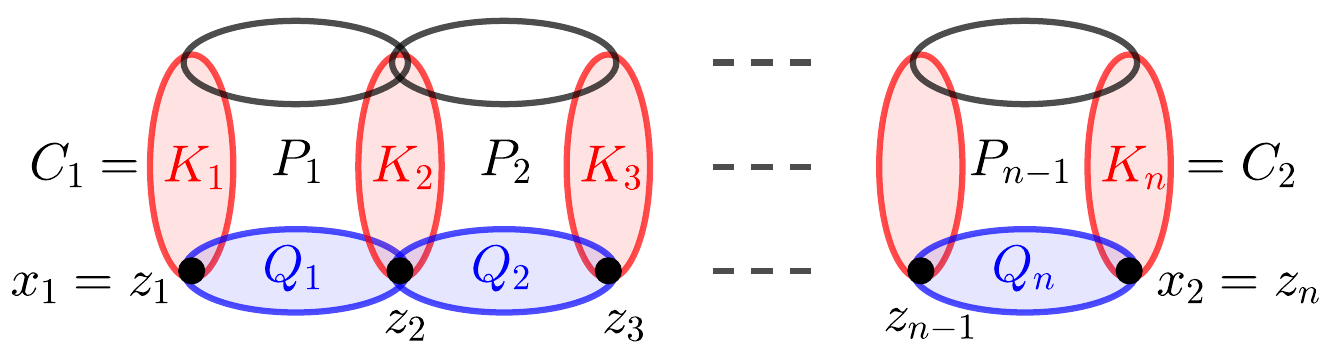}
\end{center}

\noindent
From such a data, we construct the following ladder, which shows that our edges $\{\{x_1\},C_1\}= \{\{z_1\}, K_1\}$ and $\{\{x_2\},C_2\}= \{\{z_n\}, K_n\}$ indeed belong to the same hyperplane of $\mathbb{P}(X)$.

\noindent
\begin{center}
\includegraphics[width=0.5\linewidth]{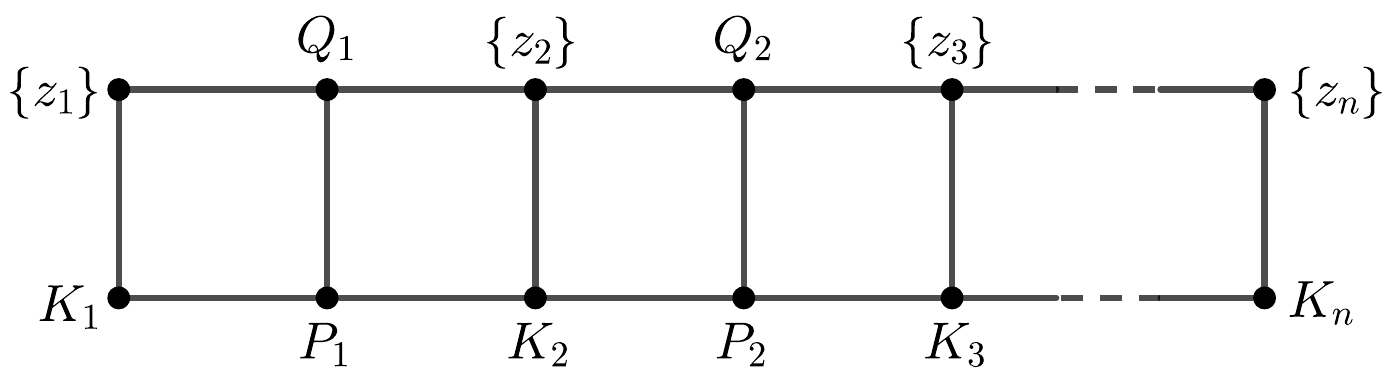}
\end{center}

\noindent
Thus, we have proved the first assertion of our theorem, i.e.\ $\bar{\Lambda}$ is a well-defined bijection. It remains to describe the halfspaces delimited by the unique hyperplane of $\mathbb{P}(X)$ labelled by a given sector $S$ of $X$. 

\begin{claim}\label{Claim:InOrNotInS}
Let $P \lessdot Q$ be two prisms of $X$ such that $\{P,Q\}$ is not labelled by $S$. Then, $P \subset S$ if and only if $Q \subset S$.
\end{claim}

\noindent
Clearly, if $P$ is not contained in $S$, then $Q$ cannot be contained in $S$ either. Now, assume that $P$ is contained in $S$. Let $R$ denote the sector of $X$ labelling $\{P,Q\}$ and let $H$ denote the hyperplane of $X$ delimiting $R$. Because $R \neq S$, necessarily $H \neq J$. Because $P$ is contained in $S$, this implies that $J$ does not cross $Q$. Therefore, $Q$ is entirely contained in a sector delimited by $J$, which has to be $S$ since $P$ is already contained in $S$. This proves Claim~\ref{Claim:InOrNotInS}. 

\medskip \noindent
Fix a clique $C$ contained in the hyperplane delimiting $S$ and let $x \in V(C)$ denote the unique vertex of $C$ contained in $S$. The edge $\{\{x\}, C\}$ is labelled by $S$. If a prism $P$ of $X$ represents a vertex of $\mathbb{P}(X)$ that lies in the same halfspaces as $\{x\}$ delimited by the hyperplane labelled by $S$ (i.e.\ the hyperplane containing $\{\{x\},C\}$), then $P$ is containted in $S$. Indeed, by connectedness of halfspaces, there must exists a path in our halfspace connecting $\{x\}$ to $P$ none of whose edges is labelled by $S$. Then, it follows from Claim~\ref{Claim:InOrNotInS} that $P \leq S$ since $\{x\} \leq S$. Similarly, if a prism $P$ of $X$ represents a vertex of $\mathbb{P}(X)$ that lies in the same halfspaces as $C$ delimited by the hyperplane labelled by $S$ (i.e.\ the hyperplane containing $\{\{x\},C\}$), then $P$ is not containted in $S$. The desired description of the two halfspaces follows. 
\end{proof}

\noindent
As a consequence of Theorem~\ref{prop:HypPrism}, notice that:

\begin{cor}\label{cor:NoHypInvPrism}
Let $G$ be a group acting on a quasi-median graph $X$. There is no hyperplane-inversion in the induced action $G \curvearrowright \mathbb{P}(X)$. 
\end{cor}

\begin{proof}
Assume for contradiction that there exists some $g \in G$ inverting some hyperplane $J$ of $\mathbb{P}(X)$. According to Proposition~\ref{prop:HypPrism}, there exists a sector $S$ of $X$ such that the two halfspaces delimited by $J$ are
$$\{\text{prisms contained in } S\} \text{ and } \{\text{prisms not contained in } S\}.$$
Since $g$ inverts $J$, it follows that, for every prism $P$ of $X$, if $P$ is contained in $S$, then $gP$ is not contained in $S$; and, if $P$ is not contained in $S$, then $gP$ is contained in $S$. Applied to the vertices of $X$, this observation shows that $g$ switches $S$ and $S^c$ in $X$. Therefore, $H$ must be a hyperplane delimiting exactly two sectors, $g$ must stabilise $H$ and switch its two sectors. Let $e$ be an edge in $H$ (which is also a clique, because the fact that $H$ delimits exactly two sectors implies that all the cliques in $H$ are edges). Then $ge$ also belongs to $H$. But then, $e$ is not contained in $S$ and $ge$ neither, hence a contradiction.
\end{proof}

\noindent
We conclude this section by recording a technical lemma for future use.

\begin{lemma}\label{lem:SepHypPrism}
Let $X$ be a quasi-median graph. For all vertices $x,y \in V(X)$, the hyperplanes separating $\{x\}$ and $\{y\}$ in $\mathbb{P}(X)$ are the hyperplanes labelled by the sectors containing one of $x,y$ but not the other.
\end{lemma}

\begin{proof}
Let $z_1:=x, z_2, \ldots, z_{n-1},z_n:=y$ be a geodesic in $X$. For every $1 \leq i \leq n-1$, let $C_i$ denote the clique of $X$ containing the edge $\{z_i,z_{i+1}\}$. It follows from Proposition~\ref{prop:DistPoly} that:

\begin{fact}\label{fact:GeodPX}
The path
$$\{z_1\}, C_1, \{z_2\}, C_2, \{z_3\}, \ldots, \{z_n\}$$
in $\mathbb{P}(X)$ is a geodesic.
\end{fact}

\noindent
Then, it follows from Claim~\ref{claim:GeodPX} that the hyperplanes of $\mathbb{P}(X)$ separating $\{x\}$ and $\{y\}$ are exactly the hyperplanes given by the labels of the edges $\{\{z_i\}, C_i\}$, $\{C_i, \{z_{i+1}\}$, $1 \leq i \leq n-1$. 

\medskip \noindent
For every $1 \leq i \leq n-1$, let $J_i$ denote the hyperplane of $X$ containing $C_i$. Notice that $J_1, \ldots, J_{n-1}$ are exactly the hyperplanes separating $x$ and $y$. For every $1 \leq i \leq n-1$, the edge $\{\{z_i\},C_i\}$ is labelled by the sector delimited by $J_i$ that contains $z_i$ but not $z_{i+1}$, or equivalently $x$ but not $y$; and the edge $\{C_i,\{z_{i+1}\}\}$ is labelled by the sector delimited by $J_i$ that contains $z_{i+1}$ but not $z_i$, or equivalently $y$ but not $x$. 

\medskip \noindent
The desired description of the hyperplanes of $\mathbb{P}(X)$ separating $\{x\}$ and $\{y\}$ follows. 
\end{proof}

\section{Minimal actions on quasi-median graphs}\label{section:BigMin}

\noindent
In Theorem~\ref{thm:BigIntro}, given a group acting on a quasi-median graph, we relate the number of sectors delimited by a hyperplane with the number of coends of its stabiliser. In order to have such a relation, orbits must visit non-trivially all the sectors of the quasi-median graph. In other words, some kind of minimality for the action is required. 

\medskip \noindent
In Section~\ref{section:Min}, we focus on convex-minimality. We show that every action can be turned into a convex-minimal action (Proposition~\ref{prop:ConvexMinimal}) and we prove a characterisation of convex-minimal actions (Proposition~\ref{prop:WhenConvMin}). 

\medskip \noindent
However, a stronger form of convex-minimality will be required in order to prove Theorem~\ref{thm:BigIntro}. Namley, given a group $G$ acting on a quasi-median graph $X$, we want the actions of $G$ on $X$ and on its graph of prisms $\mathbb{P}(X)$ to be both convex-minimal. In Section~\ref{section:StrongConvMin}, we characterise this condition (Proposition~\ref{prop:WhenStrongConvMin}) and we prove that it is automatically satisfied by monohyp actions (Theorem~\ref{thm:MonoHypConvMin}). This is the main result of this section.

\subsection{Convex-minimality}\label{section:Min}

\noindent
Recall that, given a group $G$ acting on a quasi-median graph $X$, the action is \emph{convex-minimal} if $X$ is the only non-empty $G$-invariant convex subgraph of $X$. We start by showing that every action on a quasi-median graph can be easily turned into a convex-minimal action:

\begin{prop}\label{prop:ConvexMinimal}
Let $G$ be a finitely generated group acting on a quasi-median graph $X$. There exists $x \in V(X)$ such that $G \curvearrowright \mathrm{ConvHull}(G \cdot x)$ is convex-minimal.
\end{prop}

\begin{proof}
First, we claim that, for every $x \in V(X)$, $G \curvearrowright \mathrm{ConvHull}(G \cdot x)$ has only finitely many $G$-orbits of sectors.

\medskip \noindent
Let $S$ be a sector of $\mathrm{ConvHull}(G\cdot x)$. As a consequence of Proposition~\ref{prop:MultiSectorConvex}, $S$ separates $G \cdot x$, i.e.\ there exist $g,h \in G$ such that $gx \in S$ but $hx \notin S$. Fix a finite symmetric generating set $S \subset G$ and write $g^{-1}h$ as a product of generators, say $s_1 \cdots s_n$. Along the sequence
$$g\cdot x, gs_1 \cdot x, \ldots, gs_1 \cdots s_n \cdot x= h\cdot x,$$
we can find two successive vertices such that only the first one belongs to $S$, say $gs_1 \cdots s_i \cdot x \in S$ but $gs_1 \cdots s_is_{i+1} \cdot x \notin S$. This amounts to saying that $x$ belongs to $(gs_1 \cdots s_i)^{-1}S$ but not $s_{i+1}\cdot x$. Thus, we have proved that every sector has a translate in 
$$\bigcup\limits_{s \in S} \{\text{sector containing $x$ but not $xs$}\}.$$
This set has cardinality $\leq |S| \cdot \max \{ d(x,xs), s \in S\}< \infty$. This proves our claim.

\medskip \noindent
Now, choose $x \in V(X)$ such that $G \curvearrowright \mathrm{ConvHull}(G \cdot x)$ has the smallest possible number of $G$-orbits of sectors. If this action is not convex-minimal, there exists a proper $G$-invariant convex subgraph $Y \leq \mathrm{ConvHull}(G \cdot x)$. It follows from Proposition~\ref{prop:MultiSectorConvex} that there is a sector in $\mathrm{ConvHull}(G \cdot x)$ that is disjoint from $Y$, which implies that, given an arbitrary vertex $y \in V(Y)$, the action of $G$ on $\mathrm{ConvHull}(G \cdot y) \leq Y$ has fewer $G$-orbits of sectors than $G \curvearrowright \mathrm{ConvHull}(G \cdot x)$, contradicting our choice of $x$. 
\end{proof}

\noindent
Next, we prove the following convenient characterisation of convex-minimal actions:

\begin{prop}\label{prop:WhenConvMin}
Let $G$ be a group acting on a quasi-median graph $X$. The action $G \curvearrowright X$ is convex-minimal if and only if, for all sectors $S \leq X$ and vertex $x \in V(X)$, $G \cdot x \cap V(S) \neq \emptyset$. 
\end{prop}

\begin{proof}
Assume that $G \curvearrowright X$ is not convex-minimal, i.e.\ there exists a proper $G$-invariant convex subgraph $Y \leq X$. As a consequence of Proposition~\ref{prop:MultiSectorConvex}, there exists a sector $S$ disjoint from $Y$. Fix a vertex $y \in V(Y)$, the orbit $G \cdot y \subset V(Y)$ must be disjoint from~$S$. 

\medskip \noindent
Conversely, assume that there exist a sector $S \leq X$ and a vertex $x \in V(X)$ such that $G \cdot X$ is disjoint from $S$. It follows from Proposition~\ref{prop:MultiSectorConvex} that $\mathrm{ConvHull}(G \cdot x)$ is disjoint from $S$, providing a proper convex subgraph of $X$ that is $G$-invariant. Thus, $G \curvearrowright X$ is not convex-minimal. 
\end{proof}

\subsection{Strong convex-minimality}\label{section:StrongConvMin}

\noindent
Given a group $G$ acting on a quasi-median graph $X$, it may be noticed that, even though the action $G \curvearrowright X$ is convex-minimal, the action $G \curvearrowright \mathbb{P}(X)$ induced on the graph of prisms may not be convex-minimal. This is the case, for instance, when $X$ is a clique on which $G$ acts transitively. We start by characterising precisely when the induced action on the graph of prisms is convex-minimal. Comparing with Proposition~\ref{prop:WhenConvMin}, this condition can be thought of as a natural strenghtening of convex-minimality.  

\begin{prop}\label{prop:WhenStrongConvMin}
Let $G$ be a group acting on a quasi-median graph $X$. The induced action $G \curvearrowright \mathbb{P}(X)$ is convex-minimal if and only if, for all sector $S \leq X$ and prism $P \leq X$, $P$ has a $G$-translate contained in $S$.
\end{prop}

\begin{proof}
The combination of Proposition~\ref{prop:WhenConvMin} with the description of the sectors in the graph of prisms given by Proposition~\ref{prop:HypPrism} implies that $G \curvearrowright \mathbb{P}(X)$ is convex-minimal if and only if, for all sector $S \leq X$ and prism $P \leq X$, there exist $g_1,g_2 \in G$ such that $g_1P \subset S$ and $g_2P \nsubseteq S$. This is equivalent to the condition saying that every prism has a $G$-translate in every sector. 
\end{proof}

\noindent
Finally, we prove the main result of this section, namely:

\begin{thm}\label{thm:MonoHypConvMin}
Let $G$ be a group acting on a quasi-median graph $X$. If $G \curvearrowright X$ is monohyp, then $G \curvearrowright \mathbb{P}(X)$ is convex-minimal. 
\end{thm}

\begin{proof}
Our goal is to prove the convex-minimality of $G \curvearrowright \mathbb{P}(X)$ thanks to Proposition~\ref{prop:WhenStrongConvMin}. So let $S\leq X$ be a sector and $P \leq X$ a prism. We need to find some $g \in G$ such that $gP \subset S$, or equivalently that $gP \cap S^c = \emptyset$. If $P \cap S^c= \emptyset$, then there is nothing to prove, so, from now on, we assume that $P \cap S^c \neq \emptyset$. We distinguish two cases.

\medskip \noindent
\underline{Case 1:} Assume that $P$ is contained in $S^c$. Notice that, if we denote by $J$ the hyperplane delimiting $S$ and $\mathcal{H}(P)$ the set of the hyperplanes crossing $P$, then we can write
$$\bigcap\limits_{g \in G} gS^c = \bigcap\limits_{g \in G \text{ such that } gJ \notin \mathcal{H}(P)} gS^c \cap \bigcap\limits_{g \in G \text{ such that } gJ \in \mathcal{H}(P)} gS^c.$$
If $g \in G$ is such that $gJ \notin \mathcal{H}(P)$ and if $P \nsubseteq gS^c$, then $P \cap gS^c = \emptyset$, or equivalently $g^{-1}P \cap S^c=\emptyset$. Thus, we have found a $G$-translate of $P$ disjoint from $S^c$, as desired. So let us assume that $P \subset gS^c$ for every $g \in G$ such that $gJ \notin \mathcal{H}(P)$. Consequently,
$$P \leq \bigcap\limits_{g \in G \text{ such that } gJ \notin \mathcal{H}(P)} gS^c.$$
Thus, we have proved that
$$\bigcap\limits_{g \in G} gS^c \geq P \cap \bigcap\limits_{g \in G \text{ such that } gJ \in \mathcal{H}(P)} gS^c.$$
In other words, $\bigcap_{g \in G} gS$ contains an intersection of cosectors in $P$. But this intersection has to be empty, since otherwise $\bigcap_{g \in G} gS^c$ would provide a proper convex subgraph of $X$ that is $G$-invariant, contradicting the convex-minimality of $G \curvearrowright X$. The only possibility is that, in our intersection of cosectors, we find all the cosectors of some hyperplane, say $gJ$. This implies that $\mathrm{stab}(gJ)$ permutes transitively the cosectors delimited by $gJ$, which amounts to saying that $\mathrm{stab}(gJ)$ permutes transtively the sectors delimited by $gJ$. Since the action $G \curvearrowright X$ is hyperplane-transitive, we deduce that the stabiliser of any hyperplane permutes transitively the sectors it delimits. 

\medskip \noindent
Since $P$ is contained in $S^c$, necessarily it is not crossed by $J$, i.e.\ $P$ is contained in some sector delimited by $J$. But now we know that there exists some $g \in \mathrm{stab}(J)$ that sends this sector to $S$, hence $gP \leq S$.

\medskip \noindent
\underline{Case 2:} Assume that $P$ is not contained in $S^c$. Since $P$ intersects $S^c$ by assumption, necessarily $P$ is crossed by the hyperplane $J$ delimiting $S$. If there exists $g \in G$ such that $gP$ is not crossed by $J$, then either $gP$ is disjoint from $S^c$, and we are done, or $gP$ is contained in $S^c$. In the latter case, we can apply the previous case in order to find some $h \in G$ such that $hgP$ is disjoint from $S^c$. So let us assume that, for every $g \in G$, $gP$ is crossed by $J$. This amounts to saying that, for every $g \in G$, $P$ is crossed by $gJ$. Since $G \curvearrowright X$ is hyperplane-transitive, this implies that $P$ is crossed by all the hyperplanes of $X$. As a consequence, $X$ must have only finitely many hyperplanes. So it is a bounded quasi-median graph, contradicting the assumption that $G \curvearrowright X$ has unbounded orbits. 
\end{proof}

\section{From quasi-median graphs}\label{section:FromQM}

\noindent
This section is dedicated to the proof of the following half of Theorem~\ref{thm:BigIntro}. 

\begin{thm}\label{thm:FromQM}
Let $G$ be a finitely generated group, $H \leq G$ a subgroup, and $p \in \mathbb{N}$, $q \in \mathbb{N} \cup \{\omega\}$. Assume that $G$ admits an $H$-monohyp action on a quasi-median graph $X$ such that a hyperplane stabilised by $H$ delimits $\geq q$ sectors and $\geq p$ $H$-orbits of sectors. Then, $e(G,H) \geq p$ and $\tilde{e}(G,H) \geq q$. 
\end{thm}

\begin{proof}
Fix a basepoint $o \in V(X)$ and a hyperplane $J$ stabilised by $H$. Denote by $\mathscr{S}$ the set of the sectors delimited by $J$. By assumption, we know that $\mathscr{S}$ has cardinality $\geq q$ and can be partitioned into $n \geq p$ $H$-orbits, say $\mathscr{S} = \mathscr{S}_1 \sqcup \cdots \sqcup \mathscr{S}_n$. For all $S \in \mathscr{S}$ and $1 \leq i \leq n$, we set
$$A_S:= \{ g \in G \mid go \in S\} \text{ and } B_i:= \bigcup\limits_{S \in \mathscr{S}_i} A_S= \left\{ g \in G \mid go \in \bigcup\limits_{S \in \mathscr{S}_i} S \right\}.$$
Our goal is to show that the $A_S$ verify Proposition~\ref{prop:AlmostInCoarseSep} and that the $B_i$ verify Proposition~\ref{prop:AlmostInvCodim}. The latter assertion will be a straightforward consequence of the former assertion, so we first focus on the $A_S$. 

\medskip \noindent
We know from Proposition~\ref{prop:HypPrism} that the hyperplanes of the graph of prisms $\mathbb{P}(X)$ are naturally labelled by the sectors of $X$. Given an $1 \leq i \leq n$, let $\mathbb{P}_i(X)$ denote the hyperplane-collapse of $\mathbb{P}(X)$ that collapses all the hyperplanes not labelled by sectors from $G \cdot \mathscr{S}_i$. We denote by $\pi_i$ the canonical projection $\mathbb{P}(X) \to \mathbb{P}_i(X)$. 

\begin{claim}\label{claim:DistPPi}
For all $S \in \mathscr{S}_i$ and $g \in G$, 
$$d(\pi_i(o), g \pi_i(o)) = | \mathrm{stab}(S) \backslash (A_S \triangle A_Sg^{-1}) |.$$
\end{claim}

\noindent
It follows from Lemmas~\ref{lem:DistCollapse} and~\ref{lem:SepHypPrism} that $d(\pi_i(o),g \pi_i(o))$ equals
$$\# \{ hR \mid R \in \mathscr{S}_i, \text{ either $o \in hR$ and $go \notin hR$ or $o \notin hR$ and $go \in hR$} \}.$$
Hence
$$\begin{array}{lcl} d(\pi_i(o),g\pi_i(o)) & = & \# \{ hS \mid h^{-1} \in A_S \backslash A_Sg^{-1} \cup A_Sg^{-1} \backslash A_S\} \\ \\ & = & \# \{ h \mathrm{stab}(S) \mid h^{-1} \in A_S \triangle A_Sg^{-1} \} \\ \\ & = & \# \{ \mathrm{stab}(S) h \mid h \in A_S \triangle A_Sg^{-1} \}. \end{array}$$
This concludes the proof of Claim~\ref{claim:DistPPi}.

\medskip \noindent
Given an arbitrary $S \in \mathscr{S}$, let us verify that $A_S$ is $H$-infinite. If not, then $A_S$ must be $\mathrm{stab}(S)$-finite. Indeed, if $A_S$ is $H$-finite, then we know that any sufficiently large subset $T \subset A_S$ must contain two distinct elements in the same $H$-orbit, say $a,b \in T$ satisfy $a=hb$ for some $h \in H$. Since $h$ stabilises the hyperplane delimiting $S$ and since $h$ sends $ho \in S$ to $hbo = ao \in S$, necessarily $h \in \mathrm{stab}(S)$. Thus, $A_S$ is indeed $\mathrm{stab}(S)$-finite. It follows from Claim~\ref{claim:DistPPi} that the $G$-orbit of $\pi_i(o)$ is bounded in $\mathbb{P}_i(X)$, where $1 \leq i \leq n$ is the index satisfying $S \in \mathscr{S}_i$. It follows from Corollary~\ref{cor:NoHypInvPrism} and Proposition~\ref{prop:FixedPointMedian} that $G$ fixes a vertex in $\mathbb{P}_i(X)$, say $v$. Then, we deduce from Lemma~\ref{lem:PreimageConvex} that $G$ stabilises the convex subgraph $\pi_i^{-1}(v) \subsetneq \mathbb{P}(X)$, proving that $G$ does not act convex-minimally on $\mathbb{P}(X)$, a contradiction according to Theorem~\ref{thm:MonoHypConvMin} since $G \curvearrowright X$ is monohyp by assumption.

\medskip \noindent
So far, we have proved that all the $A_S$ are $H$-infinite. It is clear that the $A_S$ are pairwise disjoint. In order to apply Proposition~\ref{prop:AlmostInCoarseSep}, it remains to verify that $A_S \cap A_S^c g$ is $H$-finite for all $S \in \mathscr{S}$ and $g \in G$. 

\begin{claim}\label{claim:DistQMSum}
Fix arbitrary sectors $S_1 \in \mathscr{S}_1, \ldots, S_n \in \mathscr{S}_n$. For every $g \in G$, the equality
$$d(o,go)= \sum\limits_{i=1}^n | \mathrm{stab}(S_i) \backslash (A_{S_i} \triangle A_{S_i}g^{-1}) |$$
holds in $\mathbb{P}(X)$. 
\end{claim}

\noindent
Our claim follows immediately from Claim~\ref{claim:DistPPi} and Corollary~\ref{cor:DistancesCollapses}.

\medskip \noindent
Given an arbitrary $g \in G$, it follows from Claim~\ref{claim:DistQMSum} that
$$|\mathrm{stab}(S) \backslash (A_S \backslash A_S g)| \leq d(o,g^{-1}o) < \infty.$$
Thus, $A_S \backslash A_Sg$ is $\mathrm{stab}(S)$-finite, and a fortiori $H$-finite since $\mathrm{stab}(S)$ is a subgroup of~$H$. 

\medskip \noindent
We are finally ready to apply Proposition~\ref{prop:AlmostInCoarseSep} and to conclude that $\tilde{e}(G,H) \geq |\mathscr{S}| \geq q$.

\medskip \noindent
Next, let us deduce from what we have done so far that Proposition~\ref{prop:AlmostInvCodim} applies to the $B_i$. It is clear that $B_1, \ldots, B_n$ are pairwise disjoint. Also, given an index $1 \leq i \leq n$, notice that $B_i:= \bigcup_{S \in \mathscr{S}_i} A_S$ is $H$-invariant (since $\mathscr{S}_i$ is an $H$-orbit of sectors) and that it is $H$-infinite (since we alread proved that each $A_S$ is $H$-infinite). Finally, given a $g \in G$, we have
$$B_i \cap B_i^c g = \left( \bigcup\limits_{S \in \mathscr{S}_i} A_S \right) \cap \left( \bigcup\limits_{S \in \mathscr{S}_i} A_S g \right)^c = \left( \bigcup\limits_{S \in \mathscr{S}_i} A_S \right) \cap \left( \bigcap\limits_{S \in \mathscr{S}_i} A_S^c g \right),$$
hence
$$B_i \cap B_i^c g \subset \bigcup\limits_{S \in \mathscr{S}_i} (A_S \cap A_S^c g).$$
Since $\mathscr{S}_i$ is an $H$-orbit, it follows that
$$H \backslash ( B_i \cap B_i^cg ) = H \backslash (A_S \cap A_S^c g),$$
which we already know to be finite. We conclude that Proposition~\ref{prop:AlmostInvCodim} does apply, proving that $e(G,H) \geq n \geq p$, as desired. 
\end{proof}

\section{Quasi-cubulation}\label{section:QuasiCubulation}

\noindent
After Theorem~\ref{thm:FromQM}, our goal is to show that every finitely generated group with a coarsely separating subgroup acts on some quasi-median graph. This will be proved in the next section, based on the construction we describe now. In a nutshell, we generalise the cubulation of wallspaces, i.e.\ the construction of a median graph from a collection of bipartitions, to the construction of a quasi-median graph from a collection of arbitrary partitions. See Appendix~\ref{section:Buneman} for more background.  

\begin{definition}
A \emph{space with characters} $(X,\mathfrak{C})$ is the data of a set $X$ and a collection $\mathfrak{C}$ of partitions of $X$ such that:
\begin{itemize}
	\item every $\mathscr{C}\in \mathfrak{C}$ contains at least two elements;
	\item no $\mathscr{C} \in \mathfrak{C}$ contains the empty set.
\end{itemize}
An element $\mathscr{C} \in \mathfrak{C}$ is a \emph{character} and an element $C \in \mathscr{C}$ is a \emph{clade}.
\end{definition}

\noindent
We emphasize that, here, $\mathfrak{C}$ is a set and not multi-set, i.e.\ we do not allow repetitions in $\mathfrak{C}$. 

\medskip \noindent
Our goal is to associate a quasi-median graph to every space with characters. Vertices will be specific \emph{selectors}:

\begin{definition}
Let $(X,\mathfrak{C})$ be a space with characters. A \emph{selector} $\sigma$ is a map $\mathscr{C} \to \mathfrak{P}(X)$ satisfying $\sigma(\mathscr{C}) \in \mathscr{C}$ for every $\mathscr{C} \in \mathfrak{C}$. 
\end{definition}

\noindent
In other words, a selector chooses a clade inside every character. Vertices of our quasi-median graphs will be selectors satisfying a specific property. In order to define this property, we need the following definition:

\begin{definition}
Let $(X,\mathfrak{C})$ be a space with characters and $\mathscr{C}$ a character. An \emph{extension} of a clade $C \in \mathscr{C}$ is a union
$$\bigcup\limits_{D \in \mathscr{C} \backslash \{D_0\}} D$$
for some $D \in \mathscr{C}$ distinct from $C$. 
\end{definition}

\noindent
In other words, an extension of a clade is the complement of a distinct clade coming from the same character. We emphasize that, if a character has more than two clades, then a clade has several possible extensions.

\begin{definition}\label{def:Coherent}
Let $(X,\mathfrak{C})$ be a space with characters. A selector $\sigma$ is \emph{coherent} if, for all characters $\mathscr{C},\mathscr{D}\in \mathfrak{C}$, $\sigma(\mathscr{C})$ and $\sigma(\mathscr{D})$ do not have disjoint extensions.
\end{definition}

\noindent
Notice that, if $(X,\mathfrak{C})$ is a wallspace, i.e.\ every character is bipartition, then a selector $\sigma$ is coherent if and only if $\sigma(\mathscr{C})\cap \sigma(\mathscr{D}) \neq \emptyset$ for all characters $\mathscr{C},\mathscr{D} \in \mathfrak{C}$. Thus, our definition of coherent selectors extends the definition of ultrafilters (also called orientations) from wallspaces.

\medskip \noindent
We are now ready to define the quasi-median graph that we associate to every space with characters. 

\begin{definition}
Let $(X,\mathfrak{C})$ be a space with characters. Its \emph{quasi-cubulation} $\overline{\mathrm{QM}}(X,\mathfrak{C})$ is the graph
\begin{itemize}
	\item whose vertices are the coherent selectors,
	\item and whose edges connect two selectors whenever they differ on a single character.
\end{itemize}
\end{definition}

\noindent
Thus, given a coherent selector $\sigma$ and a character $\mathscr{C}\in \mathfrak{C}$, the neighbours of $\sigma$ in the quasi-cubulation are the
$$[\sigma,C] : \mathscr{D} \mapsto \left\{ \begin{array}{cl} C & \text{if } \mathscr{D}= \mathscr{C} \\ \sigma(\mathscr{D}) & \text{otherwise} \end{array} \right.$$
for $C \in \mathscr{C}$. In the sequel, we will use this convenient notation repeatedly.  

\medskip \noindent
Let us verify that our quasi-cubulation does yield quasi-median graphs, as claimed.

\begin{thm}\label{thm:QuasiCubulation}
Every connected component of the quasi-cubulation $\overline{\mathrm{QM}}(X,\mathfrak{C})$ of a space with characters $(X,\mathfrak{C})$ is quasi-median. Moreover, for all coherent selectors $\mu$ and $\nu$, 
$$d_{\overline{\mathrm{QM}}(X,\mathfrak{C})} (\mu,\nu)= \# \{ \mathscr{C} \in \mathfrak{C} \mid \mu(\mathscr{C}) \neq \nu ( \mathscr{C}) \}.$$
In particular, two coherent selectors belong to the same component of $\overline{\mathrm{QM}}(X,\mathfrak{C})$ if and only if they disagree only on finitely many characters. 
\end{thm}

\begin{proof}
We start by proving the second assertion of our theorem, i.e.\ we compute distances in the quasi-cubulation.

\medskip \noindent
Let $\mu$ and $\nu$ be two coherent selectors. For convenience, we write $\mu \triangle \nu$ the set of characters on which $\mu$ and $\nu$ disagree. It is clear from the definition of $\overline{\mathrm{QM}}(X,\mathfrak{C})$ that 
$$d_{\overline{\mathrm{QM}}(X,\mathfrak{C})}(\mu,\nu) \geq  \# \mu \triangle \nu.$$
In order to prove the reverse inequality, we assume that $\mu \triangle \nu$ is finite and we construct a path of length $\#\mu \triangle \nu$ connecting $\mu$ to $\nu$ in $\overline{\mathrm{QM}}(X,\mathfrak{C})$. For this, we start by choosing a character $\mathscr{D} \in \mu \triangle \nu$ such that:
\begin{itemize}
	\item $\mu(\mathscr{D})$ is $\subseteq$-maximal in $\{ \mu(\mathscr{C}) \mid \mathscr{C} \in \mu \triangle \nu\}$;
	\item if there exists at least one other $\mathscr{C} \in \mu \triangle \nu$ such that $\mu(\mathscr{D})= \nu(\mathscr{c})$, we want $\mathscr{D}$ not to be a bipartition.
\end{itemize}
Notice that, if $\mathscr{C}_1,\mathscr{C}_2 \in \mathfrak{C}$ are two distinct characters such that $\mu(\mathscr{C}_1)= \nu(\mathscr{C}_2)$, then at least one of $\mathscr{C}_1,\mathscr{C}_2$ is not a bipartition, since otherwise we would have
$$\mathscr{C}_1= \{ \mu(\mathscr{C}_1), \mu(\mathscr{C}_1)^c \} = \{\mu(\mathscr{C}_2), \mu(\mathscr{C}_2)^c \} = \mathscr{C}_2.$$
Therefore, our second item above makes sense. Now, we claim that $[\mu,\nu(\mathscr{D})]$ is coherent. Otherwise, there exists a character $\mathscr{C} \neq \mathscr{D}$ such that $\mu(\mathscr{C})$ and $\nu(\mathscr{D})$ have disjoint extensions. Since $\nu$ is coherent, necessarily $\mathscr{C} \in \mu \triangle \nu$. And, since $\mu$ is coherent, the extension of $\nu(\mathscr{C})$ disjoint from an extension of $\mu(\mathscr{D})$ must be the complement of $\mu(\mathscr{C})$, hence $\mu(\mathscr{D}) \subset \mu(\mathscr{C})$. By maximality of $\mu(\mathscr{D})$, this implies that $\mu(\mathscr{D})= \mu(\mathscr{C})$. Notice that, since $\mu(\mathscr{D})$ has an extension disjoint from the extension $\mu(\mathscr{C})^c$ of $\nu(\mathscr{C})$, necessarily $\mu(\mathscr{D})$ is an extension of $\mu(\mathscr{D})$, which amounts to saying that $\mathscr{D}$ is a bipartition. We get a contradiction with our choice of $\mathscr{D}$.

\medskip \noindent
Thus, we have found a character $\mathscr{D} \in \mu \triangle \nu$ such that $[\mu,\nu(\mathscr{D})]$ is a coherent selector, is adjacent to $\mu$ in $\overline{\mathrm{QM}}(X,\mathfrak{C})$, and disagrees with $\nu$ on $<\# \mu \triangle \nu$ characters. By iterating the argument, we obtain a path of length $\leq \# \mu \triangle \nu$ connecting $\mu$ to $\nu$ in $\overline{\mathrm{QM}}(X,\mathfrak{C})$, proving that
$$d_{\overline{\mathrm{QM}}(X,\mathfrak{C})}(\mu, \nu) \leq \# \mu \triangle \nu,$$
as desired. 

\medskip \noindent
Now, given a connected component $\mathrm{QM}$ of $\overline{\mathrm{QM}}(X,\mathfrak{C})$, our goal is to prove that $\mathrm{QM}$ is a quasi-median graph by applying Theorem~\ref{thm:LocalGlobal}. 

\medskip \noindent
First, we prove that $\mathrm{QM}^{\square\triangle}$ is simply connected. So let $\gamma$ be a loop in $\mathrm{QM}$. Each edge of $\gamma$ amounts to modifying the clade that a given selector chooses for some character. Since, when traveling along $\gamma$, we are eventually back at our initial selector, the clade of each character is modified at least twice (or none). Fix a shortest path in $\gamma$ whose first and last edges correspond to modifying the clade of the same character. In other words, consider a subpath
$$\sigma, [\sigma,A], [\sigma,A,C_1], \ldots, [\sigma,A,C_1,\ldots, C_n], [\sigma,A,C_1, \ldots, C_n, B]$$
where $A$ and $B$ are two distinct clades of the same character and where $A,C_1, \ldots, C_n$ are clades of pairwise distinct characters. Let $\mathscr{C}$ denote the character containing $A$ and $B$; and, for every $1 \leq i \leq n$, let $\mathscr{C}_i$ denote the character containing $C_i$. 

\begin{claim}\label{claim:SimplyConnected}
For every $0 \leq i \leq n$, the selector $\sigma_i:=[\sigma,C_1, \ldots, C_i]$ is coherent.
\end{claim}

\noindent
We argue by induction over $i$. Of course, we already know that $\sigma_0=\sigma$ is coherent. So let $1 \leq i \leq n$ be an index such that $\sigma_{i-1}$ is coherent. Since $\sigma_i= [\sigma_{i-1},C_i]$, proving that $\sigma_i$ is coherent amounts to proving that, for every character $\mathscr{D} \neq \mathscr{C}_i$, the clades $\sigma_{i-1}(\mathscr{D})$ and $C_i$ do not have disjoint extensions. 

\medskip \noindent
Notice that, since $\mathscr{C}$ is distinct from the $\mathscr{C}_j$, $[[\sigma_{i-1},A],C_i]$ coincides with $[\sigma, A, C_1, \ldots, C_i]$, and in particular is coherent. Therefore, for every $\mathscr{D} \neq \mathscr{C}_i,\mathscr{C}$, the clades
$$[\sigma_{i-1},A](\mathscr{D})=\sigma_{i-1}(\mathscr{D})$$
and $C_i$ do not have disjoint extensions. Thus, it only remains to verify that $\sigma_{i-1}(\mathscr{C})$ and $C_i$ do not have disjoint extensions. If this is the case, then $\mathscr{C}$ has only one clade all of whose extensions intersect every extension of $C_i$. But, since $[\sigma,A,C_1, \ldots, C_n]$ is coherent, 
$$[\sigma,A,C_1, \ldots, C_n](\mathscr{C})=A \text{ and } [\sigma,A,C_1, \ldots, C_n](\mathscr{C}_i)= C_i$$
do not have disjoint extensions. (Here, we use the fact that $\mathscr{C}$ is distinct from the $\mathscr{C}_j$ and the fact that $\mathscr{C}_i$ is distinct from the $\mathscr{C}_j$ for $j \neq i$.) Similarly, since $[\sigma,B,C_1, \ldots, C_n]$ is coherent, 
$$[\sigma,B,C_1, \ldots, C_n](\mathscr{C})=B \text{ and } [\sigma,B,C_1, \ldots, C_n](\mathscr{C}_i)= C_i$$
do not have disjoint extensions. Hence a contradiction since $A \neq B$. This concludes the proof of Claim~\ref{claim:SimplyConnected}.

\medskip \noindent
\begin{minipage}{0.35\linewidth}
\includegraphics[width=\linewidth]{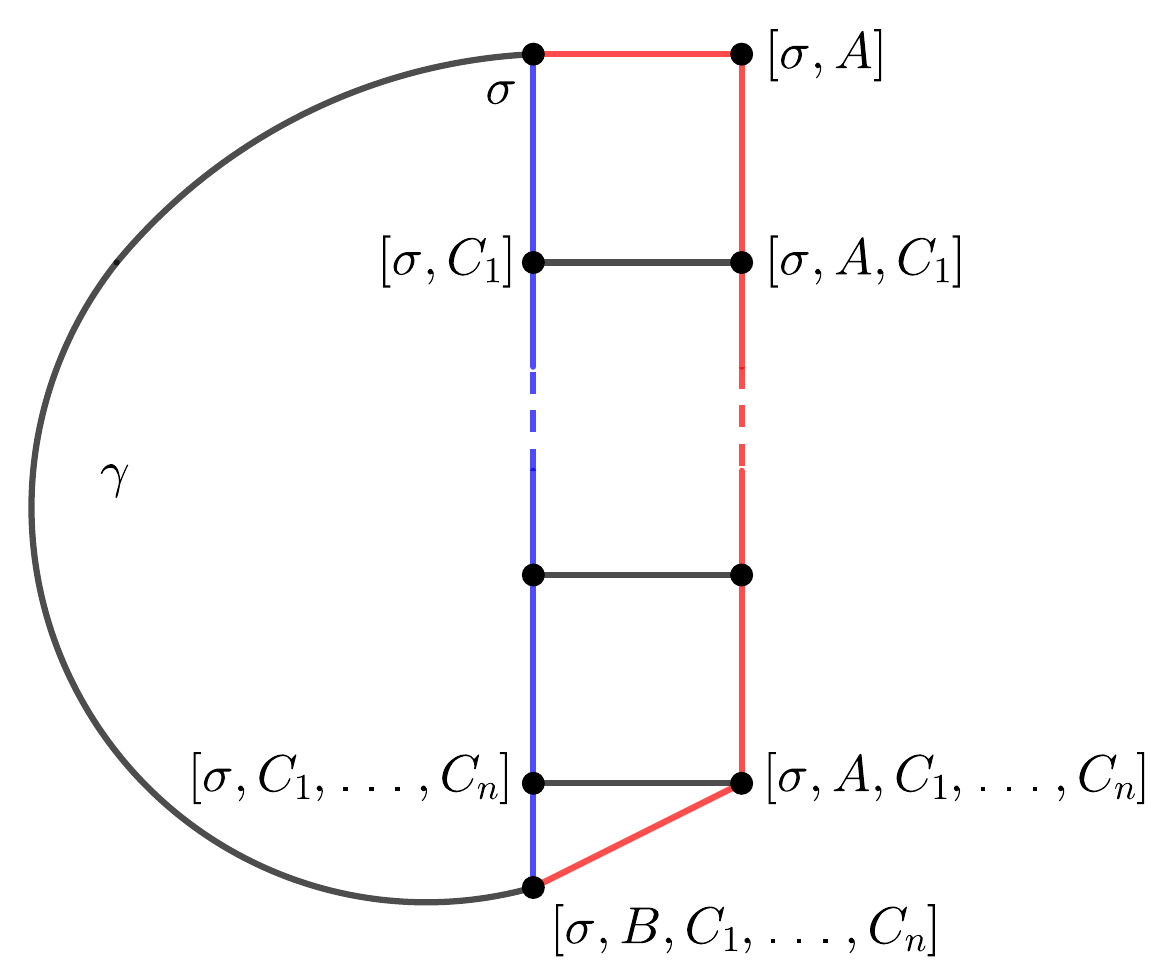}
\end{minipage}
\begin{minipage}{0.63\linewidth}
As illustrated by the figure on the left, notice that the vertices $\sigma$, $[\sigma,C_1]$, $\ldots$, $[\sigma, C_1, \ldots, C_n]$ define a ladder with $[\sigma,A, C_1]$, $\ldots$, $[\sigma,A,C_1, \ldots, C_n]$ in $\mathrm{QM}$, and that $[\sigma,C_1, \ldots, C_n]$ and $[\sigma,B,C_1, \ldots, C_n]$ either agree or define a triangle with $[\sigma,A,C_1, \ldots, C_n]$ (depending on whether or not $B$ coincides with $\sigma(\mathscr{C})$). 
\end{minipage}

\medskip \noindent
We record our construction for future use:

\begin{claim}\label{claim:SameHypConstruct}
For all pairwise distinct characters $\mathscr{C},\mathscr{C}_1, \ldots, \mathscr{C}_n \in \mathfrak{C}$ and all clades $A \neq B \in \mathscr{C}, C_1 \in \mathscr{C}_1, \ldots, C_n \in \mathscr{C}_n$, if
$$[\sigma,A], [\sigma,A,C_1], \ldots, [\sigma,A,C_1, \ldots, C_n]$$
is a path connecting the edges $\{\sigma, [\sigma,A]\}$ and $\{[\sigma,B,C_1, \ldots,C_n], [\sigma, A,C_1, \ldots, C_n]\}$, then there exists a ladder connecting $\{\sigma, [\sigma,A]\}$ to $\{[\sigma,C_1, \ldots, C_n], [\sigma,A,C_1, \ldots, C_n]\}$. Moreover, the vertices $[\sigma,C_1, \ldots, C_n]$ and $[\sigma,B,C_1, \ldots, C_n]$ either agree or define a triangle with $[\sigma,A,C_1, \ldots, C_n]$
\end{claim}

\noindent
Thus, replacing in $\gamma$ the subpath
$$\sigma, [\sigma,A], [\sigma,A,C_1], \ldots, [\sigma,A,C_1, \ldots, C_n], [\sigma,B ,C_1, \ldots, C_n]$$
with
$$\sigma, [\sigma,C_1], \ldots, [\sigma,C_1, \ldots, C_n], [\sigma,B,C_1, \ldots, C_n]$$
shortens its length (by one or two) and does not modify its homotopy class in $\mathrm{QM}^{\square\triangle}$. After finitely many iterations of this argument, we conclude that $\gamma$ is homotopy equivalent to a single point in $\mathrm{QM}^{\square\triangle}$, as desired.

\medskip \noindent
In order to conclude that $\mathrm{QM}$ is a quasi-median graph thanks to Theorem~\ref{thm:LocalGlobal}, it remains to verify that $\mathrm{QM}$ does not contain induced copies of $K_{2,3}$ or $K_4^-$, that it satisfies the $3$-cube condition, and that it satisfies the $3$-prism condition.

\medskip \noindent
Assume for contradiction that $\mathrm{QM}$ contains an induced copy of $K_{2,3}$. Such a copy must be isometrically embedded, so the two vertices of degree three can be written as $\sigma$ and $[\sigma,A,B]$ for some clades $A \in \mathscr{A}$ and $B \in \mathscr{B}$ coming from distinct characters. The three vertices of degree two in $K_{2,3}$ must be selectors that differ from both $\sigma$ and $[\sigma,A,B]$ on a single character. So at least two of these selectors differ from $\sigma$ only at $\mathscr{A}$ or only at $\mathscr{B}$, and consequently must be adjacent in $\mathrm{QM}$, proving that our copy of $K_{2,3}$ is not induced.

\medskip \noindent
Consider a copy of $K_4^-$ in $\mathrm{QM}$. Let $\sigma$ be a coherent selector representing a vertex of degree three in $K_4^-$. Its neighbours can be written as $[\sigma,A]$, $[\sigma,B]$, and $[\sigma,C]$ for some characters $\mathscr{A},\mathscr{B},\mathscr{C} \in \mathfrak{C}$ and clades $A \in \mathscr{A}$, $B \in \mathscr{B}$, $C \in \mathscr{C}$. Up to renaming our characters, assume that $[\sigma,B]$ is the neighbour of degree three of $\sigma$ in $K_4^-$. Because $[\sigma,A]$ and $[\sigma,B]$ are adjacent, they must differ on a single character, which imposes that $\mathscr{A}= \mathscr{B}$. Similarly, we must have $\mathscr{B}= \mathscr{C}$. Thus, $[\sigma,A]$ and $[\sigma,C]$ differ only on the character $\mathscr{A}=\mathscr{C}$, and consequently must be adjacent, proving that our copy of $K_4^-$ is not induced.

\medskip \noindent
Consider an induced copy of $Q_3^-$ in $\mathrm{QM}$. 

\medskip \noindent
\begin{minipage}{0.3\linewidth}
\includegraphics[width=0.95\linewidth]{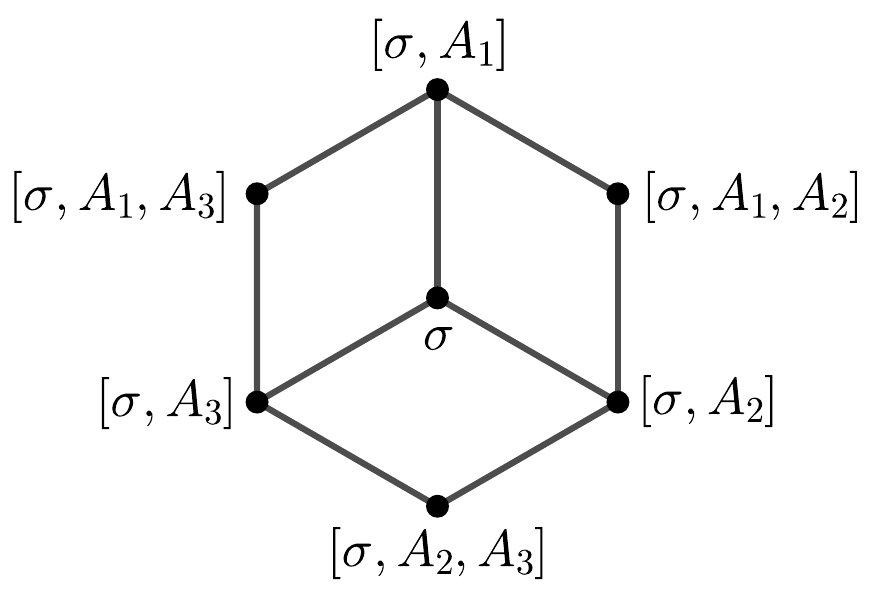}
\end{minipage}
\begin{minipage}{0.68\linewidth}
Fix a coherent selector $\sigma$, characters $\mathscr{A}_1,\mathscr{A}_2, \mathscr{A}_3 \in \mathfrak{C}$, and clades $A_1 \in \mathscr{A}_1$, $A_2 \in \mathscr{A}_2$, $A_3 \in \mathscr{A}_3$ such that our copy of $Q_3^-$ can be described as in the left figure. If the selector $[\sigma,A_1,A_2,A_3]$ is coherent, then it provides the missing vertex of the copy of the $3$-cube $Q_3$ containing $Q_3^-$ that we are looking for. 
\end{minipage}

\medskip \noindent
But, if $[\sigma,A_1,A_2,A_3]$ is not coherent, then one of $[\sigma, A_1,A_2]$, $[\sigma,A_1,A_3]$, $[\sigma,A_2,A_3]$ must not be coherent as well, which is not the case since we know that they already represent vertices in $\mathrm{QM}$.

\medskip \noindent
Finally, consider a copy of the house graph $H$ in $\mathrm{QM}$. 

\medskip \noindent
\begin{minipage}{0.3\linewidth}
\includegraphics[width=0.95\linewidth]{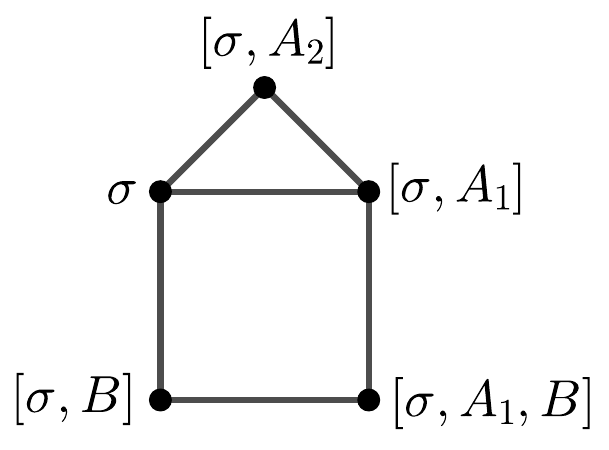}
\end{minipage}
\begin{minipage}{0.68\linewidth}
Fix a selector $\sigma$, characters $\mathscr{A},\mathscr{B} \in \mathfrak{C}$, and clades $A_1,A_2 \in \mathscr{A}$, $B\in \mathscr{B}$ sucht that our copy of $H$ can be described as in the left figure. If the selector $[\sigma,B,A_2]$ is coherent, then it provides the missing vertex of the copy of the $3$-prism $K_2 \times K_3$ that we are looking for.
\end{minipage}

\medskip \noindent
If $[\sigma, A_2,B]$ is not coherent, then $A_2$ and $B$ must have disjoint extensions, say $A_2^+$ and $B^+$. Because $[\sigma,B]$ is coherent, $\sigma(\mathscr{A})$ cannot have an extension disjoint from some extension of $B$, so $\sigma(\mathscr{A})$ cannot be contained in $A_2^+$, which amounts to saying that $\sigma(\mathscr{A})$ coincides with the unique clade of $\mathscr{A}$ not contained in $A_2^+$. Similarly, because $[\sigma,A_1,B]$ is coherent, $A_1$ must coincide with the unique clade of $\mathscr{A}$ not contained in $A_2^+$. Hence $\sigma(\mathscr{A})= A_1$. This implies that $\sigma=[\sigma,A_1]$, which is impossible since our copy of $H$ is induced. 

\medskip \noindent
Thus, we have proved that Theorem~\ref{LocalGlobal} applies, and we conclude, as desired, that $\mathrm{QM}$ is a quasi-median graph.
\end{proof}

\paragraph{Choosing a component.} According to Theorem~\ref{thm:QuasiCubulation}, the quasi-cubulation of a space with characters is a disjoint union of quasi-median graphs. Now, we need to choose one connected component in a rather canonical way. The motivation is that, given a group $G$ acting on the underlying set of a space with characters $(X,\mathfrak{C})$, if $\mathfrak{C}$ is $G$-invariant then we would like $G$ to act on our canonical component of the quasi-cubulation $\overline{\mathrm{QM}}(X,\mathfrak{C})$. In full generality, this is not possible. Nevertheless, under some reasonable assumption, there will be a canonical choice.

\begin{definition}
A space with characters $(X,\mathfrak{C})$ is \emph{locally finite} if the following condition is satisfied:
\begin{itemize}
	\item[] for all $x,y \in X$ and for all but finitely many characters $\mathscr{C}\in \mathfrak{C}$, $x$ and $y$ belong to the same clade of $\mathscr{C}$.
\end{itemize}
\end{definition}

\noindent
As we shall see, there is a natural choice for a component of the quasi-cubulation of a locally finite space with characters. It is based on the following family of selectors:

\begin{definition}
Let $(X,\mathfrak{C})$ be a space with characters. A selector $\sigma$ is \emph{pointed} if there exists some $x \in X$ such that $x \in \sigma(\mathscr{C})$ for every character $\mathscr{C} \in \mathfrak{C}$. 
\end{definition}

\noindent
Notice that pointed selectors are clearly coherent. Moreover, it follows from Theorem~\ref{thm:QuasiCubulation} that, if our space with characters is locally finite, then all the pointed selectors belong to the same component of the quasi-cubulation. This allows us to define:

\begin{definition}
Given a locally finite space with characters $(X,\mathfrak{C})$, we denote by $\mathrm{QM}(X,\mathfrak{C})$ the component of the quasi-cubulation $\overline{\mathrm{QM}}(X,\mathfrak{C})$ that contains the pointed selectors. 
\end{definition}

\noindent
Notice that, if a group $G$ acts on the underlying set of a locally finite space with characters $(X,\mathfrak{C})$ and that $\mathfrak{C}$ is $G$-invariant, then $G$ sends pointed selectors to pointed selectors, and consequently preserves $\mathrm{QM}(X,\mathfrak{C})$. Thus, we obtain, as desired, an action of $G$ on a quasi-median graph, namely $\mathrm{QM}(X,\mathfrak{C})$.

\medskip \noindent
We conclude this section by describing the hyperplanes and sectors in our quasi-median graphs. In our next statement, use the fact that, given a space with characters $(X,\mathfrak{C})$, the edges of $\overline{\mathrm{QM}}(X,\mathfrak{C})$ are naturally labelled by $\mathfrak{C}$. Indeed, an edge connect two coherent selectors that differ precisely on one character, and we can use this character as a label of our edge. 

\begin{prop}\label{prop:HypInQC}
Let $(X,\mathfrak{C})$ be a locally finite space with characters. The map
$$\Lambda : \left\{ \begin{array}{ccc} E(\mathrm{QM}(X,\mathfrak{C})) & \to & \mathfrak{C} \\ e & \mapsto & \text{label of } e \end{array} \right.$$
induces a bijection from the set of hyperplanes of $\mathrm{QM}(X,\mathfrak{C})$ to $\mathfrak{C}$. Moreover, for every character $\mathscr{C}\in \mathfrak{C}$, the sectors delimited by the unique hyperplane of $\mathrm{QM}(X,\mathfrak{C})$ labelled by $\mathscr{C}$ are the subgraphs induced by
$$\{ \xi \in V(\mathrm{QM}(X,\mathfrak{C})) \mid \xi(\mathscr{C})=C\} \text{ for } C \in \mathscr{C}.$$ 
\end{prop}

\begin{proof}
First, let us verify that $\Lambda$ induces a well-defined map from the set of hyperplanes of $\mathrm{QM}(X,\mathfrak{C})$ to $\mathfrak{C}$. In other words:

\begin{claim}\label{claim:SameLabel}
In $\overline{\mathrm{QM}}(X,\mathfrak{C})$, the edges of a $3$-cycles are all labelled by the same character and two opposite edges in a $4$-cycle are labelled by the same character.
\end{claim}

\noindent
The vertex-set of a $3$-cycle can be written as $\{\sigma, [\sigma,A], [\sigma,B]\}$ for some coherent selector $\sigma$, characters $\mathscr{A},\mathscr{B}\in \mathfrak{C}$, and clades $A \in \mathscr{A}$, $B \in \mathscr{B}$. Notice that, since $[\sigma,A]$ and $[\sigma,B]$ represent adjacent vertices, as selectors they must disagree on exactly one character, hence $\mathscr{A}= \mathscr{B}$. Thus, the three edges of our $3$-cycle are labelled by $\mathscr{A}= \mathscr{B}$. This proves the first assertion of our claim.

\medskip \noindent
Next, consider a $4$-cycle in $\overline{\mathrm{QM}}(X,\mathfrak{C})$. If it is not induced, then it must be contained in a complete subgraph (as a consequence of the fact that a quasi-median graph has no induced copy of $K_4^-$) and we conclude from the previous observation that all its edges have the same label. So, from now on, assume that our $4$-cycle is induced, and in particular isometrically embedded. It follows from Theorem~\ref{thm:QuasiCubulation} that two opposite vertices in $K_{2,3}$ can be represented as $\sigma$ and $[\sigma, A,B]$ for two distinct characters $\mathscr{A},\mathscr{B} \in \mathfrak{C}$ and some clades $A \in \mathscr{A}, B \in \mathscr{B}$. The two remaining vertices of our cycle must differ just on a single character from both $\sigma$ and $[\sigma,A,B]$. Clearly, there are only two possibilities: $[\sigma,A]$ and $[\sigma,B]$. Thus, our $4$-cycle is $(\sigma, [\sigma,A], [\sigma,A,B], [\sigma,B])$. Indeed, two opposite edges are labelled by the same character. 

\medskip \noindent
This concludes the proof of Claim~\ref{claim:SameLabel}.

\medskip \noindent
Thus, we have proved that $\Lambda$ induces a map $\bar{\Lambda}$ from the set of hyperplanes of $\mathrm{QM}(X,\mathfrak{C})$ to $\mathfrak{C}$. Let us verify that $\bar{\Lambda}$ is injective. In other words:

\begin{claim}\label{claim:IFFsameHypLabel}
Two edges in a common component of $\overline{\mathrm{QM}}(X,\mathfrak{C})$ belong to the same hyperplane if and only if they have the same label.
\end{claim}

\noindent
We already know from Claim~\ref{claim:SameLabel} that two edges of the same hyperplane have the same label. So let $\{\sigma, [\sigma,A]\}$ and $\{\varsigma, [\varsigma, B]\}$ be two edges where $\sigma,\varsigma$ are two coherent selectors belonging to the same component of $\overline{\mathrm{QM}}(X,\mathfrak{C})$ and where $A,B$ are two clades from the same character $\mathscr{C} \in \mathfrak{C}$. Fix a geodesic
$$[\sigma,A], [\sigma,A,C_1], \ldots, [\sigma,A,C_1, \ldots, C_n] = [\varsigma,B]$$
connecting $[\sigma,A]$ to $[\varsigma,B]$, where $C_1 \in \mathscr{C}_1, \ldots, C_n \in \mathscr{C}_n$ are clades coming from characters $\mathscr{C}_1, \ldots, \mathscr{C}_n \in \mathfrak{C}$. As a consequence of Theorem~\ref{thm:QuasiCubulation}, the fact that our path is a geodesic amounts to saying that the characters $\mathscr{C},\mathscr{C}_1, \ldots, \mathscr{C}_n$ are pairwise distinct. Then, we deduce from Claim~\ref{claim:SameHypConstruct} that our edges $\{\sigma, [\sigma,A]\}$ and $\{\varsigma, [\varsigma, B]\}$ belong to the same hyperplane, as desired. This concludes the proof of Claim~\ref{claim:IFFsameHypLabel}.

\medskip \noindent
Finally, let us verify that $\bar{\Lambda}$ is surjective. So, given a character $\mathscr{C} \in \mathfrak{C}$, we need to find a hyperplane of $\mathrm{QM}(X,\mathfrak{C})$ labelled by $\mathscr{C}$. Let $x,y \in X$ be two points that belong to two distinct clades of $\mathscr{C}$. Consider a path connecting the selectors $\sigma_x$ and $\sigma_y$ respectively pointed at $x$ and $y$. Since $\sigma_x$ and $\sigma_y$ disagree on $\mathscr{C}$, along our path we can find two consecutive selectors that disagree on $\mathscr{C}$. The corresponding edge must be labelled by $\mathscr{C}$, and we conclude that the hyperplane containing this edges is labelled by $\mathscr{C}$. This concludes the proof of the fact that $\bar{\Lambda}$ is bijective.

\medskip \noindent
Given a character $\mathscr{C} \in \mathfrak{C}$, it remains to describe the sectors delimited by the hyperplane of $\mathrm{QM}(X,\mathfrak{C})$ labelled by $\mathscr{C}$. But it is clear that the subgraphs induced by the
$$\{ \xi \in V(\mathrm{QM}(X,\mathfrak{C})) \mid \xi(\mathscr{C})=C\} \text{ for } C \in \mathscr{C}$$ 
are the maximal subgraphs of $\mathrm{QM}(X,\mathfrak{C})$ that have no edge labelled by $\mathscr{C}$. Equivalently, according to Claim~\ref{claim:IFFsameHypLabel}, they are the maximal subgraphs not crossed by the hyperplane labelled by $\mathscr{C}$, which amounts to saying that they are the sectors delimited by our hyperplane labelled by $\mathscr{C}$. 
\end{proof}

\section{From coarsely separating subgroups}\label{section:FromCoarseSep}

\noindent
This section is dedicated to the proof of the remaining half of Theorem~\ref{thm:BigIntro}.

\begin{thm}\label{thm:FromCoarseSep}
Let $G$ be a finitely generated group, $H \leq G$ a subgroup, and $p \in \mathbb{N}$, $q \in \mathbb{N}_{\geq 2} \cup \{\omega\}$. If $e(G,H) \geq p$ and $\tilde{e}(G,H) \geq q$, then $G$ admits an $H$-monohyp action on a quasi-median graph such that a hyperplane stabilised by $H$ delimits $\geq q$ sectors and $\geq p$ $H$-orbits of sectors. 
\end{thm}

\begin{proof}
We know from the definition of the number of coends and from Proposition~\ref{prop:CoarseSepCodim} that there exists $L \geq 3$ such that $G \backslash H^{+L}$ contains $\geq q$ deep connected components and $\geq p$ $H$-orbits of deep connected components. Let $H_+$ denote the complement of the union of all the deep components of $G\backslash H^{+L}$. Notice that $H_+$ is $H$-finite, and consequently must be contained in some neighbourhood of $H$, say $H^{+R}$ for some $R \geq L$. Set
$$\mathscr{C}:= \{\text{deep component of } G \backslash H^{+L}\} \cup \{ H, H_+\backslash H \}.$$
Notice that:

\begin{claim}\label{claim:StabC}
The $G$-stabiliser of $\mathscr{C}$ is $H$.
\end{claim}

\noindent
Because $L \geq 3$, $H$ is the only element of $\mathscr{C}$ that, as a subset of $G$, does not contain any ball of radius $1$. (Notice that, if $H$ contains a ball of radius $1$, then it must contain the generating set given by the word-metric of $G$ under consideration, hence $H=G$, contradicing the assumption $\tilde{e}(G,H) \geq q \geq 2$.) Consequently, $\mathrm{stab}(\mathscr{C})$ must stabilise $H$, i.e.\ $\mathrm{stab}(\mathscr{C}) \leq H$. Conversely, it is clear that $H$ stabilises $\mathscr{C}$. This proves Claim~\ref{claim:StabC}.

\medskip \noindent
Now, we set $\mathfrak{C}:= \{ g \mathscr{C} \mid g \in G\}$. Following Section~\ref{section:QuasiCubulation}, our goal is to show that $(G,\mathfrak{C})$ is a locally finite space with characters and that the natural action of $G$ on a convex-minimal subgraph of $\mathrm{QM}(G,\mathfrak{C})$ satisfies the conditions we are looking for. We start by verifying that:

\begin{claim}\label{Claim:LocallyFinite}
The space with characters $(G,\mathfrak{C})$ is locally finite. 
\end{claim}

\noindent
Let $x,y \in G$ be two points and let $g_1 \mathscr{C},\ldots, g_n \mathscr{C}$ be characters for which $x$ and $y$ belong to two distinct clades. Given an index $1 \leq i \leq n$, we can find two distinct $C_1,C_2 \in \mathscr{C}$ such that $x \in g_iC_1$ and $y \in g_iC_2$. If $C_1$ or $C_2$ belongs to $\{H, H_+\backslash H\}$, then it is clear that $g_iH^{+R}$ intersects the ball $B:=B(x,d(x,y))$. Otherwise, $C_1$ and $C_2$ are two distinct deep components of $G\backslash H^{+L}$. Since $g_iH^{+L}$ separates $g_iC_1$ and $g_iC_2$, some point along a geodesic connecting $x$ to $y$, which necessarily belong to $B$, will belong to $g_iH^{+L}$ as well, hence $B \cap g_iH^{+L} \neq \emptyset$. In any case, we know that $g_iH^{+R}$ intersects $B$. Thus, the cosets $g_1H, \ldots, g_nH$ all intersect $B^{+R}$. By local finiteness of $G$, we know that, if $n$ is too large, then there must exist $i \neq j$ such that $g_iH \cap g_jH \neq \emptyset$, which implies that $g_iH=g_jH$, and finally $g_i \mathscr{C}=g_j \mathscr{C}$ as a consequence of Claim~\ref{claim:StabC}. This concludes the proof of Claim~\ref{Claim:LocallyFinite}.

\medskip \noindent
Since $(G,\mathscr{C})$ is locally finite, and because $\mathfrak{C}$ is $G$-invariant, we obtain a quasi-median graph $\mathrm{QM}:= \mathrm{QM}(G,\mathscr{C})$ on which $G$ acts, as defined in Section~\ref{section:QuasiCubulation}. According to Proposition~\ref{prop:ConvexMinimal}, $\mathrm{QM}$ contains a convex subgraph $\mathrm{QM}_-$ on which $G$ acts convex-minimally. Our goal now is to verify that the action $G \curvearrowright \mathrm{QM}_-$ satisfies the conditions of our theorem. 

\medskip \noindent
We know from Proposition~\ref{prop:HypInQC} that every sector $S$ in $\mathrm{QM}$ can be described as
$$V(S)= \{ \xi \in V(\mathrm{QM}) \mid \xi(g \mathscr{C}) = gC\}$$
for some character $g \mathscr{C}$ and some clade $C \in \mathscr{C}$. Here, $C$ is either $H$, or $H_+\backslash H$, or a deep component of $G\backslash H^{+L}$. In the latter case, we refer to $S$ as a \emph{deep sector}. 

\begin{claim}\label{claim:ConvexMin}
Every deep sector of $\mathrm{QM}$ intersects every $G$-orbit. 
\end{claim}

\noindent
Let $\sigma \in V(\mathrm{QM})$ be an arbitrary selector and $S$ a deep sector of $\mathrm{QM}$. Fix an element $g \in G$ and a deep component $C$ of $G \backslash H^{+L}$ such that
$$V(S)= \{ \xi \in V(\mathrm{QM}) \mid \xi(g \mathscr{C}) = g(C \cup H)\}.$$
Up to translating $S$, we assume for convenience that $g=1$. Because $C$ is deep, we can find infinitely many $g_1,g_2, \ldots \in C$ that are pairwise distinct modulo $H$. According to Claim~\ref{claim:StabC}, this amounts to saying that the characters $g_1^{-1} \mathscr{C}, g_2^{-1} \mathscr{C}, \ldots$ are pairwise distinct. Our goal is to show that there exists $i \geq 1$ such that $\sigma(g_i^{-1} \mathscr{C})= g_i^{-1}C$.

\medskip \noindent
For every $i \geq 1$, we know that $1$ belongs to the clade $g_i^{-1}C$ of the character $g_i^{-1}\mathscr{C}$. Consequently, if $\sigma_1$ denotes the selector pointed at $1$, then $\sigma_1(g_i^{-1}\mathscr{C})=g_i^{-1}C$ for every $i \geq 1$. But, by definition of $\mathrm{QM}$, we know that $\sigma$ and $\sigma_1$ differ only on finitely many characters, hence $\sigma(g_i^{-1}\mathscr{C})= \sigma_1(g_i^{-1}\mathscr{C})=g_i^{-1}C$ for all but finitely many $i \geq 1$.

\medskip \noindent
Thus, we have proved that there exists $i \geq 1$ such that $(g_i\sigma)(\mathscr{C}) = C$, hence $g_i \sigma \in V(S)$. This concludes the proof of Claim~\ref{claim:ConvexMin}.

\begin{claim}\label{claim:UnboundedOrbits}
The action $G \curvearrowright \mathrm{QM}_-$ has unbounded orbits.
\end{claim}

\noindent
If our action has bounded orbits, it follows from Proposition~\ref{prop:FixedPointMedian} that $G$ stabilises a prism $P$ in $\mathrm{QM}_-$. 

\medskip \noindent
First, assume that $P$ is not a single vertex. Because we know from  Proposition~\ref{prop:HypInQC} that $G \curvearrowright \mathrm{QM}$ has a single orbit of hyperplanes (as $\mathscr{C}$ has single orbit of characters), this implies that all the hyperplanes of $\mathrm{QM}$ cross $P$. A fortiori, $\mathrm{QM}$ must have only finitely many hyperplanes. But we also know from Proposition~\ref{prop:HypInQC} that hyperplane-stabilisers are conjugates of $H$ (as $\mathrm{stab}(\mathscr{C})=H$ according to Claim~\ref{claim:StabC}), so it implies that $H$ must have finite index in $G$, which is a contradiction with the assumption $\tilde{e}(G,H) \geq q \geq 2$. 

\medskip \noindent
Next, assume that $P$ is reduced to a single vertex. Notice that every hyperplane of $\mathrm{QM}$ delimits $\geq q \geq 2$ deep sectors, which implies, as a consequence of Claim~\ref{claim:ConvexMin}, that the orbit $G \cdot P$ must intersect two disjoint sectors. This contradicts the fact that $P$ is a vertex fixed by $G$. 

\medskip \noindent
This concludes the proof of Claim~\ref{claim:UnboundedOrbits}.

\medskip \noindent
We are finally ready to conclude the proof of our theorem. The action $G \curvearrowright \mathrm{QM}_-$ is convex-minimal by definition, and it has unbounded orbits according to Claim~\ref{claim:UnboundedOrbits}. As already mentioned, it follows from Proposition~\ref{prop:HypInQC} that $G$ acts on $\mathrm{QM}$ with a single orbit of hyperplanes and with hyperplane-stabilisers conjugate to $H$. Necessarily, the same holds of the action on $\mathrm{QM}_-$ (since the latter is not reduced to a single vertex). Thus, the action $G \curvearrowright \mathrm{QM}_-$ is $H$-monohyp. Then, given a hyperplane $J$ in $\mathrm{QM}_-$ stabilised by $H$, it follows from the definition of $\mathscr{C}$ and Proposition~\ref{prop:HypInQC} that $J$ delimits $\geq q$ deep sectors and $\geq p$ $H$-orbits of deep sectors in $\mathrm{QM}$. We conclude from Claim~\ref{claim:ConvexMin} that $J$ delimits $\geq q$ sectors and $\geq p$ $H$-orbits of sectors in $\mathrm{QM}_-$. 
\end{proof}

\section{Proof of Theorem~\ref{thm:BigIntro} and its consequences}\label{section:Proof}

\begin{proof}[Proof of Theorem~\ref{thm:BigIntro}.]
Our theorem is an immediate consequence of Theorems~\ref{thm:FromQM} and~\ref{thm:FromCoarseSep}. 
\end{proof}

\begin{proof}[Proof of Corollary~\ref{cor:BigIntro}.]
The first item follows from Theorem~\ref{thm:BigIntro} with $q=2$. The third item follows from Theorem~\ref{thm:BigIntro} with $p=1$. The fourth item follows from Theorem~\ref{thm:BigIntro} and from Proposition~\ref{prop:ConvexMinimal}.

\medskip \noindent
The second item requires more work. Assume that $H$ is a codimension-one subgroup. We know from Theorem~\ref{thm:BigIntro} that $G$ admits an $H$-monohyp action on some quasi-median graph $X$ such that a hyperplane stabilised by $H$ delimits at least two $H$-orbits of sectors. Fix a hyperplane $J$ stabilise by $H$ and let $S_0$ be a sector delimited by $J$. Set $S_0^+:= \bigcup_{h \in H} hS_0$. Notice that, because $J$ delimits at least two $H$-orbits of sectors, $S_0^+$ is a proper subgraph of $X$. Thus, we can endow $V(X)$ with a collection of characters by setting
$$\mathscr{C}:= \{ V(S_0^+), V(X\backslash S_0^+) \} \text{ and } \mathfrak{C}:= \{ g \mathscr{C} \mid g \in G\}.$$
Notice that $\mathrm{stab}(\mathscr{C})=H$, which implies that $\mathfrak{C}$ is naturally in bijection with the set of hyperplanes of $X$. Also, notice that the space with characters $(V(X),\mathfrak{C})$ is locally finite. Indeed, for all $x,y \in V(X)$, if $g_1\mathscr{C}, \ldots, g_n\mathscr{C}$ are characters for which $x$ and $y$ belong to distinct clades, then $g_1J, \ldots, g_nJ$ are hyperplanes of $X$ separating $x$ and $y$. Consequently, there exist at most $d_X(x,y)< \infty$ characters for which $x$ and $y$ do not belong to the same clade. 

\medskip \noindent
Thus, we can consider the action of $G$ on the quasi-cubulation $\mathrm{M}:= \mathrm{QM}(V(X),\mathfrak{C})$. Because of the bijection between $\mathfrak{C}$ and the set of hyperplanes of $X$ already mentioned, it follows from the fact the $G \curvearrowright X$ is $H$-monohyp and from Proposition~\ref{prop:HypInQC} that $\mathrm{M}$ has a single $G$-orbit of hyperplanes and that hyperplane-stabilisers are conjugates of $H$. Because $H$ fixes each clade of $\mathscr{C}$, we also deduce from Proposition~\ref{prop:HypInQC} that $H$ does not contain a hyperplane-inversion. Finally, because every character has two clades, it follows again from Proposition~\ref{prop:HypInQC} that every hyperplane of $\mathrm{M}$ delimits exactly two sectors, which amounts to saying, according to Lemma~\ref{lem:WhenMedian}, that $\mathrm{M}$ is a median graph. Thus, it only remains to verify that $G$ acts on $\mathrm{M}$ with unbounded orbits.

\medskip \noindent
If $G \curvearrowright \mathrm{M}$ has bounded orbits, then, according to Proposition~\ref{prop:FixedPointMedian}, $G$ fixes some vertex, say $\sigma$. Consequently, the intersection
$$I:= \bigcap\limits_{\mathscr{D} \in \mathfrak{C}} \sigma(\mathscr{D}) \subsetneq X$$
is $G$-invariant, and we know from Proposition~\ref{prop:MultiSectorConvex} that it is convex. If we show that $I$ is non-empty, then this will contradict the convex-minimality of $G \curvearrowright X$. 

\medskip \noindent
Fix an arbitrary vertex $x \in V(X)$. Since $\sigma$ is a vertex of $\mathrm{M}$, necessarily $\sigma$ and the selector $\sigma_x$ pointed at $x$ differ on only finitely many characters, say $g_1\mathscr{C}, \ldots, g_n\mathscr{C}$. Because the $\sigma(g_i\mathscr{C})$ pairwise intersect, so do their gated hulls $\overline{\sigma(g_i\mathscr{C})}$, hence 
$$\bar{I}:= \bigcap\limits_{i=1}^n \overline{\sigma(g_i\mathscr{C})} \neq \emptyset$$
as a consequence of the Helly property satisfied by gated subsets. Let $y \in \bar{I}$ denote the gate of $x$ in $\bar{I}$. According to Proposition~\ref{prop:GatedHullMulti}, for every $1 \leq i \leq n$, either $y$ belongs to $\sigma(g_i\mathscr{C})$ or it belongs to $\overline{\sigma(g_i\mathscr{C})} \backslash \sigma(g_i\mathscr{C}) \subset N(g_i \mathscr{C})$. Up to reindexing our characters, assume that $y$ belongs to $N(g_iJ) \backslash \sigma(g_i \mathscr{C})$ for every $1 \leq i \leq p$ and that $y$ belongs to $\sigma(g_i \mathscr{C})$ for every $p<i \leq n$. 

\medskip \noindent
Notice that $g_1J, \ldots, g_pJ$ are pairwise transverse. Indeed, if $g_rJ$ and $g_sJ$ are not transverse for some $1 \leq r,s \leq p$, then $y$ lies between $g_rJ$ and $g_sJ$, so $\sigma(g_r\mathscr{C})$ and $\sigma(g_s\mathscr{C})$ must contained in the cosectors respectively delimited by $g_rJ$ and $g_sJ$ that do not contain $y$, hence $\sigma(g_r\mathscr{C}) \cap \sigma(g_s \mathscr{C})= \emptyset$, a contradiction. 

\medskip \noindent
Thus, we can apply Lemma~\ref{lem:ConstructPrism} and find a prism $P$ containing $y$ whose hyperplanes are exactly $g_1J, \ldots, g_pJ$. In $P$, we can clearly a vertex $z$ in $\sigma(g_1\mathscr{C}) \cap \cdots \cap \sigma(g_p\mathscr{C})$. We claim that $z$ belongs to our intersection $I$. This amounts to saying that $\sigma$ coincides with the selector $\sigma_z$ pointed at $z$. 

\medskip \noindent
First, let us verify that $\sigma_x \triangle \sigma_y = \{g_{p+1}\mathscr{C}, \ldots, g_n\mathscr{C}\}$. Let $g\mathscr{C} \in \sigma_x \triangle \sigma_y$. The hyperplane $gJ$ must separate $x$ and $y$. It follows from Lemma~\ref{lem:SepProj} that $gJ$ separates $x$ and $\bar{I}$, and then from Lemma~\ref{lem:InclusionOne} that $gJ$ separates $x$ and $\overline{\sigma(g_i\mathscr{C})}$ for some $1 \leq i \leq n$. Thus, $\sigma_y(g \mathscr{C})$ contains $\sigma(g_i \mathscr{C})$, which implies that $\sigma(g \mathscr{C})= \sigma_y(g\mathscr{C})$. Since $\sigma_y(g \mathscr{C}) \neq \sigma_x(g \mathscr{C})$, necessarily $g \mathscr{C} \in \sigma \triangle \sigma_x$. Moreover, since $y \in \sigma_y(g\mathscr{C})= \sigma(g\mathscr{C})$, we must have $g \mathscr{C} \in \{g_{p+1}\mathscr{C}, \ldots, g_n \mathscr{C} \}$. Conversely, Let $p+1 \leq i \leq n$. We know that $y \in \sigma(g_i \mathscr{C})$, hence $\sigma_y(g_i\mathscr{C})= \sigma(g_i \mathscr{C}) \neq \sigma_x(g_i\mathscr{C})$, and finally $g_i \mathscr{C} \in \sigma_x \triangle \sigma_y$, as desired.

\medskip \noindent
Second, it is clear by construction of $z$ that $\sigma_y \triangle \sigma_z = \{ g_1 \mathscr{C}, \ldots, g_p \mathscr{C}\}$ since, when passing from $y$ to $z$, each $\sigma_y(g_i \mathscr{C})$ is changed into $\sigma(g_i \mathscr{C})$. 

\medskip \noindent
We deduce from the previous two observations that
$$\begin{array}{lcl} \sigma \triangle \sigma_z & = & (\sigma \triangle \sigma_x) \triangle (\sigma_x \triangle \sigma_y) \triangle ( \sigma_y \triangle \sigma_z) \\ \\ & = & \{g_1\mathscr{C}, \ldots, g_n \mathscr{C}\} \triangle \{g_{p+1} \mathscr{C}, \ldots, g_n \mathscr{C}\} \triangle \{g_1 \mathscr{C} , \ldots, g_p\mathscr{C}\} = \emptyset \end{array}$$
which proves that $\sigma=\sigma_z$, as desired. 
\end{proof}

\section{Codimensionality-one versus coarse separation}\label{section:CodimVSsep}

\noindent
This section is dedicated to the proof of Theorem~\ref{thm:IntroSep} from the introduction. The fourth first items will be rather straightforward. The fifth item is more interesting, so we prove it separately. It is a consequence of the following construction:

\begin{prop}\label{prop:Example}
Let $G:= H \ast_Z S$ be an amalgam of two one-ended finitely generated groups. Assume that:
\begin{itemize}
	\item $Z$ is infinite and almost malnormal in $H$;
	\item $H$ is not coarsely separated by $Z$;
	\item $S$ has no proper finite-index subgroup.
\end{itemize}
Then, $S$ is a coarsely separating subgroup of $G$ that is not virtually a codimension-one subgroup of $G$. 
\end{prop}

\noindent
For instance, we can take for $S$ a simple group (e.g.\ Thompson's groups $T$ and $V$), for $H$ one of the many acylindrically hyperbolic groups not coarsely separable by a family of subexponential growth exhibited in \cite{CoarseSep, CoarseSepHyp, CoarseSepRAAG} (e.g.\ a uniform lattice in $\mathbb{H}^n$ for $n \geq 3$), and for $Z$ a cyclic subgroup $H$ generated by a Morse element. 

\medskip \noindent
During the proof of our proposition, the following elementary observation will be useful:

\begin{lemma}\label{lem:EasySep}
Let $X$ be a graph and $A,B,C \subset V(X)$ three subsets of vertices. If $B$ separates $A$ and $C$, then $A^{+L} \cap C \subset B^{+L}$ for every $L \geq 0$.
\end{lemma}

\begin{proof}
Fix an $L \geq 0$. Given a vertex $c \in A^{+L} \cap C$, there exists a vertex $a \in A$ satisfying $d(a,c) \leq L$. A geodesic $\gamma$ connecting $a$ to $c$ must intersect $B$. Let $b$ be a vertex in $\gamma \cap B$. Since $d(b,c) \leq d(a,c) \leq L$, we conclude that $c \in B^{+L}$.
\end{proof}

\begin{proof}[Proof of Proposition~\ref{prop:Example}.]
Let $T$ denote the Bass-Serre tree of $G$ associated to its amalgam decompotition. We think of $G$ as a tree of spaces over $T$ whose edge-spaces are the cosets of $Z$ and whose vertex-spaces are the cosets of $H$ and $S$. 

\begin{claim}\label{claim:Example}
For all $g \in G$ and $L \geq 0$, if $gZ \cap S^{+L}$ is infinite, then $g \in S$.
\end{claim}

\noindent
Notice that $g \in S$ if and only if the edge-space $gZ$ is attached to the vertex-space $S$. Therefore, if $g \notin S$, then the edge-space $gZ$ is separated from the vertex-space $S$ by some vertex-space $sH$ where $s \in S$. This implies that $gZ$ and $S$ are separated by two distinct edge-spaces $sZ$ and $shZ$ attached to $sH$, where $h \in H$. It follows from Lemma~\ref{lem:EasySep} that $S^{+L} \cap gZ \subset sZ^{+L} \cap shZ^{+L}$. But, since $Z$ is almost malnormal in $G$, we know that the coarse intersection between two discting cosets of $Z$ is coarsely bounded. Therefore, $sZ^{+L} \cap shZ^{+L}$, and a fortiori $S^{+L} \cap gZ$, must be finite. This concludes the proof of Claim~\ref{claim:Example}. 

\medskip \noindent
We deduce from Claim~\ref{claim:Example} that:

\begin{claim}\label{claim:NoSep}
The vertex-space $S$ does not coarsely separate any vertex-space.
\end{claim}

\noindent
Let $V$ be a vertex-space distinct from $S$. We distinguish two cases, depending on whether or not $V$ is adjacent to the vertex-space $S$.

\medskip \noindent
First, assume that $V$ is not adjacent to $S$. Then, $V$ is separated from $S$ by some edge-space $E$ that is not attached to $S$. It follows from Lemma~\ref{lem:EasySep} that $S^{+L} \cap V \subset E^{+L} \cap S^{+L}$ and from Claim~\ref{claim:Example} that $E^{+L} \cap S^{+L}$ is finite. Thus, $S^{+L}$ cannot coarsely separate $V$ since $V$ is one-ended. 

\medskip \noindent
Next, assume that $V$ is adjacent to $S$, i.e.\ $V=sH$ for some $s \in S$. We know from Lemma~\ref{lem:EasySep} that $S^{+L} \cap sH \subset sZ^{+L}$ since $sZ^{+L}$ separates $S$ and $sH$. But we also know by assumption that $Z$ does not coarsely separate $H$, so $S^{+L}$ cannot coarsely separate~$sH$.

\medskip \noindent
This concludes the proof of Claim~\ref{claim:NoSep}.

\medskip \noindent
Removing from the Bass-Serre tree $T$ the vertex $S$ yields a collection of connected components $\{T_s, \ s \in S\}$. Let $T_s^+ \subset G$ denote the union of all the vertex-spaces given by the vertices of $T_s$. 

\begin{claim}\label{claim:SubtreeNoSep}
For every $s \in S$, $S$ does not coarsely separate $T_s^+$. 
\end{claim}

\noindent
It follows from \cite[Proposition~4.6]{CoarseSepRAAG} that, if $S$ coarsely separates $T_s^+$, then $S$ must contain an edge-space of $T_s^+$ in a neighbourhood. But this is impossible since we know from Claim~\ref{claim:NoSep} that the coarse intersection between $S$ and any edge-space that is not attached to $S$ is bounded. This proves Claim~\ref{claim:SubtreeNoSep}.

\medskip \noindent
Of course, $S$ pairwise separates the $T_s^+$ for $s \in S$, so it follows from Claim~\ref{claim:SubtreeNoSep} that, for every $L \geq 0$, the deep connected components of $G\backslash S^{+L}$ are, up to finite Hausdorff distance, the $T_s^+$ for $s \in S$. As a consequence, $\tilde{e}(G,S)= \omega+1$ and $e(G,S)=1$ since $S$ acts transitively on the infinitely many deep components of $G \backslash S^{+L}$ for every $L \geq 0$. In particular, $S$ is a coarsely separating subgroup that is not (virtually) a codimension-one subgroup of $G$.
\end{proof}

\begin{proof}[Proof of Theorem~\ref{thm:IntroSep}.]
The first item is an immediate consequence of the definition of the number of coends and from Proposition~\ref{prop:CoarseSepCodim}. The second item is a particular case of the following observation:

\begin{claim}\label{claim:SubCodim}
If $2 \leq \tilde{e}(G,H) \leq \omega$ and $2 \leq p \leq \tilde{e}(G,H)$, then $H$ contains a finite-index subgroup $H_0$ satisfying $e(G,H_0) \geq p$.
\end{claim}

\noindent
Fix an $L \geq 0$ such that $G \backslash H^{+L}$ has $\geq p$ deep connected components. Since $\tilde{e}(G,H) \leq \omega$, we know that there are only finitely many such components. Consequently, if $H_0$ denotes the kernel of the action of $H$ on the set of deep components of $G \backslash H^{+L}$, then $H_0$ has finite index in $H$. Choosing a sufficiently large $R \geq 0$, each deep component of $G \backslash H_0^{+R}$ is contained in some deep component of $G \backslash H^{+R}$. It follows that the set of deep components of $G \backslash H_0^{+R}$ contains $\geq p$ $H_0$-orbits, hence $e(G,H_0) \geq p$, as desired. This proves Claim~\ref{claim:SubCodim}.

\medskip \noindent
The third item of our theorem can be proved similarly. We already know that, if $H$ is virtually codimension-one, then it is coarsely separating. So assume that $H$ is coarsely separating. If $H$ is codimension-one, then there is nothing to prove. So we assume that $e(G,H)=1$. Thus, fixing some $L \geq 0$ such that $G\backslash H^{+L}$ has at lest two deep components, $H$ acts transitively on the set of deep components of $G\backslash H^{+L}$. Consequently, we can find a deep component of $G\backslash H^{+L}$ and an element $h_0 \in H$ such that $h_0C \neq C$. Since $H$ is subgroup separable, there exists a finite-index subgroup $H_0 \leq H$ containing $\mathrm{stab}(C)$ but not $h_0$. Choosing a sufficiently large $R \geq 0$, each deep component of $G \backslash H_0^{+R}$ is contained in some deep component of $G \backslash H^{+R}$. By construction, a deep component contained in $C$ cannot be in the same $H_0$-orbit as a deep component contained in $h_0C$. Hence $e(G,H_0) \geq 2$, i.e.\ $H_0$ is a codimension-one subgroup of $G$.

\medskip \noindent
Now, let us verify the fourth item of our theorem. If $\tilde{e}(G,H) \leq \omega$, then we already know from the second item that $H$ contains a codimension-one subgroup of $G$. So, from now on, we assume that $\tilde{e}(G,H) = \omega+1$. Thus, there exists an $L \geq 0$ such that $G \backslash H^{+L}$ contains infinitely many deep components. Let $C$ be one of them. Consider
$$\partial C:= \{ x \in G \mid x \notin C \text{ adjacent to a point in } C \}.$$
We claim that $\mathrm{stab}_H(C)$ acts cofinitely on $\partial C$. 

\medskip \noindent
Indeed, let $P \subset \partial C$ be a finite collection of points. Our goal is to justify that, if $P$ is sufficiently big, then it contains at least two points in the same $\mathrm{stab}_H(C)$-orbit. Because $H$ acts cofinitely on $H^{+L}$, which contains $\partial C$, we know that many points of $P$ must belong to the same $H$-orbit, say $\{h_1x, \ldots, h_nx\} \subset P$ for some $x \in G$ and $h_1, \ldots, h_n \in H$ with $n$ very large. For every $1 \leq i \leq n$, $h_ix \in \partial C$ so $C \cap B(h_ix,1) \neq \emptyset$, or equivalently $h_i^{-1}C \cap B(x,1) \neq \emptyset$. Thus, $h_1^{-1}C, \ldots, h_n^{-1}C$ are deep components of $G\backslash H^{+L}$ all intersecting the ball $B(x,1)$. Since any two such component must be disjoint, the finiteness of $B(x,1)$ implies, because $n$ is very large, that there exist $1 \leq i< j \leq n$ such that $h_i^{-1}C=h_j^{-1}C$, or equivalently that $h_ih_j^{-1} \in \mathrm{stab}_H(C)$. This concludes the proof of our claim.

\medskip \noindent
Since we now know that $\mathrm{stab}_H(C)$ acts cofinitely on $\partial C$, the Hausdorff dimension between $\mathrm{stab}_H(C)$ and $\partial C$ in $G$ must be finite. Since $\partial C$ separates $C$ from the infinitely many other deep components of $G \backslash H^{+L}$, it follows that the complement of some neighbourbood $\mathrm{stab}_H(C)^{+R}$ of $\mathrm{stab}_H(C)$ has at least two deep components, one of them being $C$ (up to finite Hausdorff distance). Since $\mathrm{stab}_H(C)$ stabilises this component, it follows that $\mathrm{stab}_H(C)$ does not permute transitively the deep components of $G\backslash \mathrm{stab}_H(C)^{+R}$. We conclude from Proposition~\ref{prop:CoarseSepCodim} that $\mathrm{stab}_H(C)$ is a codimension-one subgroup of $G$.

\medskip \noindent
Finally, the fifth item of our theorem is proved by Proposition~\ref{prop:Example}.
\end{proof}

\appendix

\section{Buneman-like graphs}\label{section:Buneman}

\noindent
Let $(X,\mathscr{C})$ be a space of characters. The \emph{graph of selectors}, whose vertices are all the selectors and whose edges connect any two selectors that differ on a single character, is a disjoint union of Hamming graphs (i.e.\ products of complete graphs). It encodes $X$, thought of as the set of pointed selectors, into a graph that highlights how the characters differentiate the elements of $X$. However, the graph of selectors is just too big. The challenge of data visualisation is to find a reasonable intermediate between $X$ and the graph of selectors that still keeps track of the characters. Classical applications of such constructions include the study of phylogenetic networks (see for instance \cite{Phylogenetics}). 

\medskip \noindent
The strategy, in order to choose a smaller subgraph inside the graph of selectors, is to impose restrictions on the selectors that we allow. 

\medskip \noindent
An early illustration of this idea, when characters are bipartitions, is given by the \emph{Buneman graph} \cite{Buneman}, whose connection with median graphs is explained in \cite{MR1002234}. From the point of view of geometric group theory, the construction coincides with the \emph{cubulation} of \emph{wallspaces}, initiated in \cite{MR1347406} and further studied in \cite{Buneman,MR2197811,MR2059193}. Despite the fact that the Buneman graph was first defined for spaces with bipartitions, the definition makes sense for arbitrary spaces with characters. 

\begin{definition}
Let $(X,\mathfrak{C})$ be a space with characters. A selector $\sigma$ is \emph{Buneman-coherent} if $\sigma(\mathscr{C}_1) \cap \sigma(\mathscr{C}_2) \neq \emptyset$ for all $\mathscr{C}_1, \mathscr{C}_2 \in \mathfrak{C}$. The \emph{Buneman graph} is the graph whose vertices are the Buneman-coherent selectors and whose edges connect any two selectors that differ on a single character. 
\end{definition}

\noindent
A slightly different construction was proposed in \cite{MR1844029}.

\begin{definition}
Let $(X,\mathfrak{C})$ be a space with characters. A selector $\sigma$ is \emph{relation-coherent} if, for all $\mathscr{C}_1,\mathscr{C}_2 \in \mathfrak{C}$, either $\sigma(\mathscr{C}_1) \cap \sigma(\mathscr{C}_2) \neq \emptyset$ or $\sigma(\mathscr{C}_1) \cup \sigma(\mathscr{C}_2) = X$. The \emph{relation graph} is the graph whose vertices are the relation-coherent selectors and whose edges connect any two selectors that differ on a single character.
\end{definition}

\noindent
It should be clear from the definitions that one gets a sequence of subgraphs
$$\text{Buneman graph} \subset \text{relation graph} \subset \text{quasi-cubulation} \subset \text{graph of selectors}.$$
Investigating the relations between these graphs would be interesting. It is worth mentioning that \cite[Theorem~1]{MR1907821} characterises the so-called \emph{quasi-median hull} of the pointed selectors inside the graph of selectors, proving that a selector belongs to this quasi-median hull if and only if it is coherent in the sense of Definition~\ref{def:Coherent}, providing an alternative description of the quasi-cubulation. Thus, \cite[Corollary~1, Theorem~3]{MR1907821} characterise when the quasi-cubulation coincides with the relation graph and when it coincides with the graph of selectors. 

\medskip \noindent
From our perspective of quasi-median geometry, it is natural ask whether the Buneman graph or the relation graph can be used instead of the quasi-cubulation in order to construct quasi-median graphs. 

\medskip \noindent
However, for arbitrary spaces with characters, Buneman and relation graphs may not be even connected. Figure~\ref{BunemanDis} shows a space with characters for which the Buneman graph is connected, while the relation graph coincides with the quasi-cubulation. And Figure~\ref{RelationDis} shows a space with characterse for which the Buneman graph and the relation graph coincide but are not connected. Connectedness is not the only obstruction for being quasi-median. Figure~\ref{RelationIso} shows a space with characters for which the Buneman graph and the relation graph coincide, are connected, but are very far from being quasi-median since they do not even isometrically embed into a Hamming graph. 
\begin{figure}[h!]
\begin{center}
\includegraphics[width=0.8\linewidth]{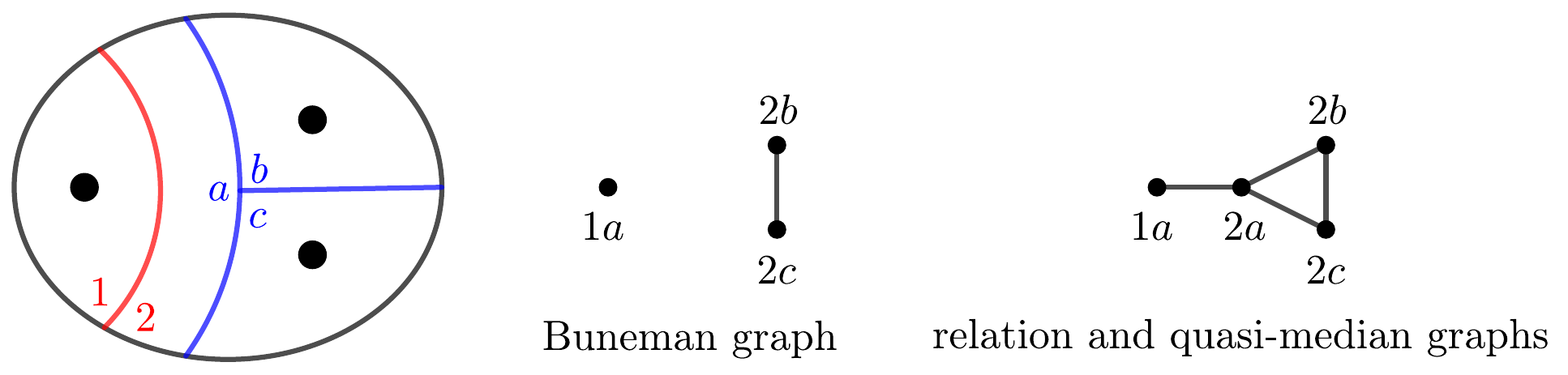}
\end{center}
\caption{A space with characters with a disconnected Buneman graph but connected relation and quasi-median graphs.}
\label{BunemanDis}
\end{figure}

\begin{figure}[h!]
\begin{center}
\includegraphics[width=0.8\linewidth]{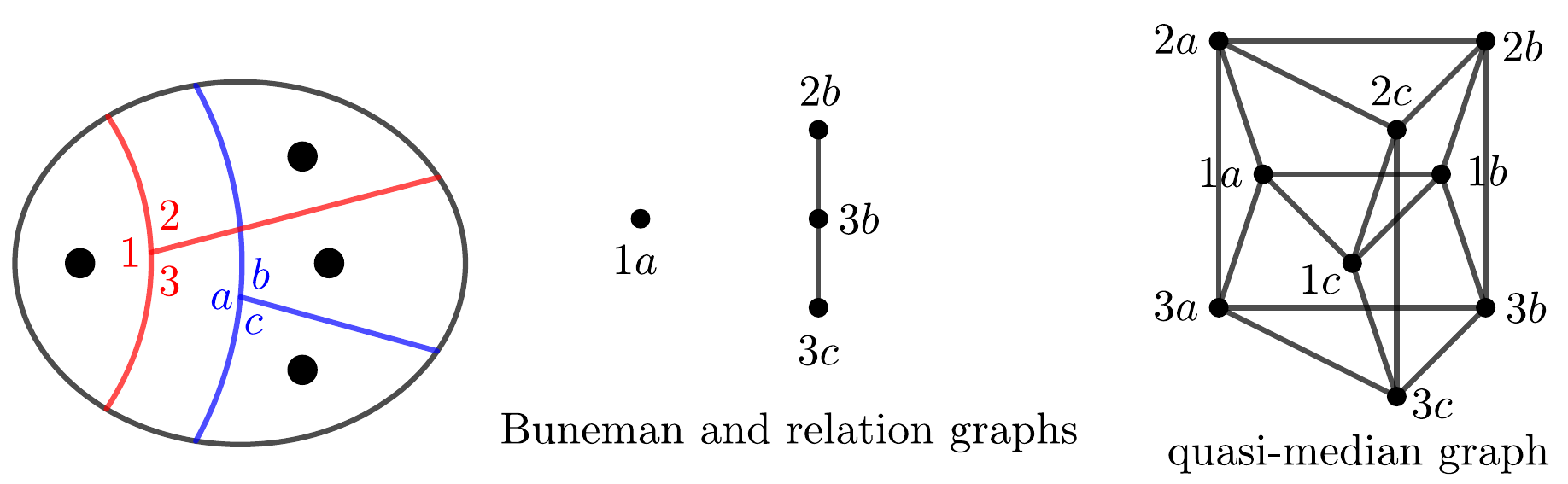}
\end{center}
\caption{A space with characters with disconnected Buneman and relation graphs but a connected quasi-median graph.}
\label{RelationDis}
\end{figure}

\begin{figure}[h!]
\begin{center}
\includegraphics[width=0.9\linewidth]{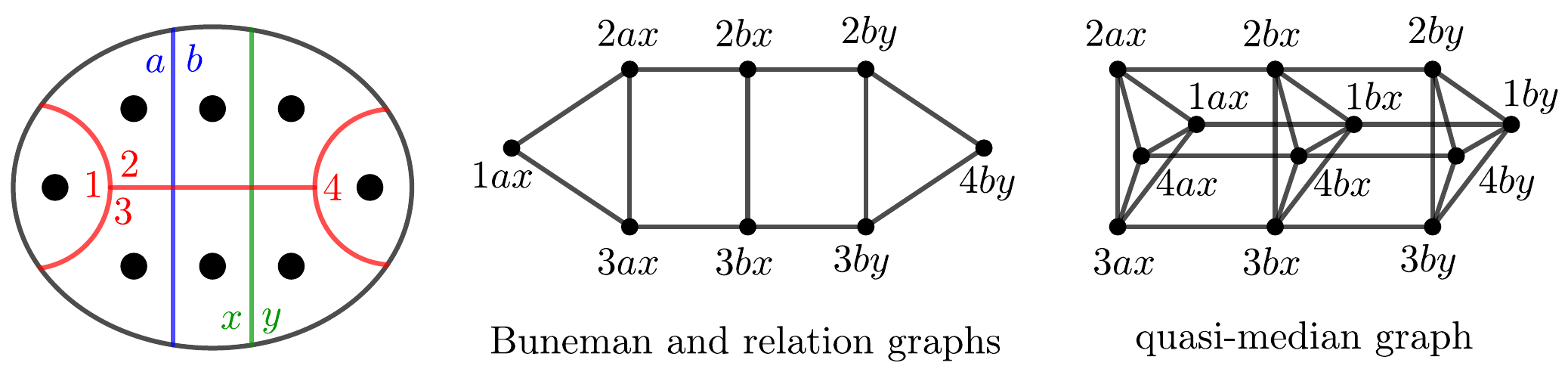}
\end{center}
\caption{A space with characters with a connected relation graph not isometrically embedded into the quasi-median graph.}
\label{RelationIso}
\end{figure}

\noindent
Nevertheless, it is worth mentioning that the Buneman graph or the relation graph may define a smaller quasi-median graph than the quasi-cubulation, as shown by Figure~\ref{SmallerQM}. This suggests that there might exist a more efficient quasi-cubulation of spaces with characters. 

\begin{question}
Is there a simple and systematic way to associate to any locally finite space with characters a quasi-median graph, which coincides with the Buneman graph as soon as the latter is already quasi-median?
\end{question}

\begin{figure}
\begin{center}
\includegraphics[width=0.8\linewidth]{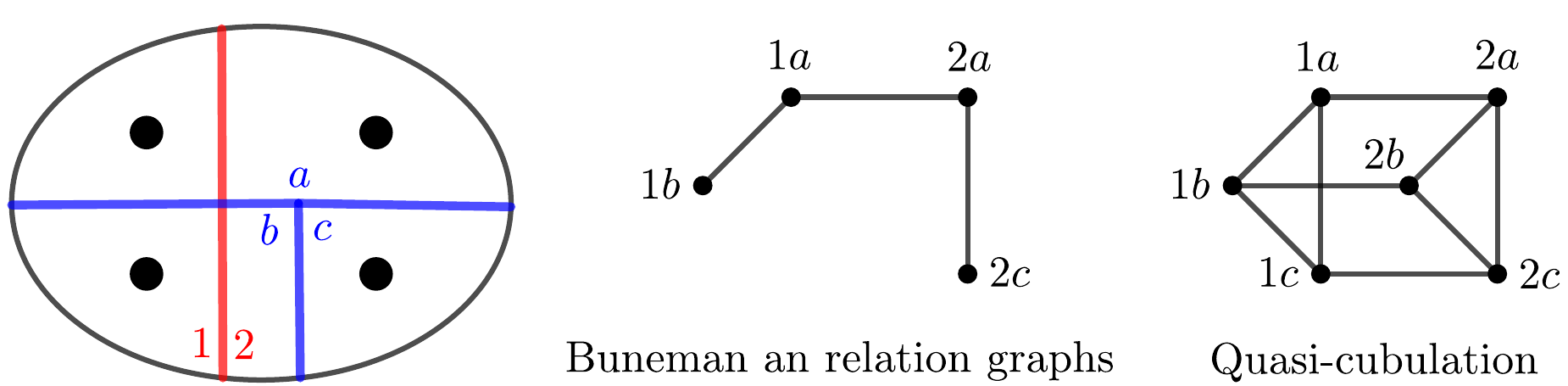}
\end{center}
\caption{A space with characters for which the Buneman graph is a smaller quasi-median graph than the quasi-cubulation.}
\label{SmallerQM}
\end{figure}

\addcontentsline{toc}{section}{References}

\bibliographystyle{alpha}
{\footnotesize\bibliography{CodimCoarseSep}}

\newcommand{\etalchar}[1]{$^{#1}$}
\begin{thebibliography}{BCC{\etalchar{+}}13}

\bibitem[Bar89]{MR1002234}
J.-P. Barth\'el\'emy.
\newblock From copair hypergraphs to median graphs with latent vertices.
\newblock {\em Discrete Math.}, 76(1):9--28, 1989.

\bibitem[BCC{\etalchar{+}}13]{MR3062742}
B.~Bre\v{s}ar, J.~Chalopin, V.~Chepoi, T.~Gologranc, and D.~Osajda.
\newblock Bucolic complexes.
\newblock {\em Adv. Math.}, 243:127--167, 2013.

\bibitem[BGT24]{CoarseSep}
O.~Bensaid, A.~Genevois, and R.~Tessera.
\newblock Coarse separation and large-scale geometry of wreath products.
\newblock {\em arxiv:2401.18025}, 2024.

\bibitem[BGT26a]{CoarseSepHyp}
O.~Bensaid, A.~Genevois, and R.~Tessera.
\newblock Coarse separation and splittings in hyperbolic groups.
\newblock {\em preprint}, 2026.

\bibitem[BGT26b]{CoarseSepRAAG}
O.~Bensaid, A.~Genevois, and R.~Tessera.
\newblock Coarse separation and splittings in right-angled {A}rtin groups.
\newblock {\em preprint}, 2026.

\bibitem[BHM02]{MR1907821}
H.-J. Bandelt, K.~Huber, and V.~Moulton.
\newblock Quasi-median graphs from sets of partitions.
\newblock {\em Discrete Appl. Math.}, 122(1-3):23--35, 2002.

\bibitem[BMW94]{MR1297190}
H.-J. Bandelt, H.~Mulder, and E.~Wilkeit.
\newblock Quasi-median graphs and algebras.
\newblock {\em J. Graph Theory}, 18(7):681--703, 1994.

\bibitem[Bow02]{MR1867372}
B.~Bowditch.
\newblock Splittings of finitely generated groups over two-ended subgroups.
\newblock {\em Trans. Amer. Math. Soc.}, 354(3):1049--1078, 2002.

\bibitem[Bun71]{Buneman}
P.~Buneman.
\newblock The recovery of trees from measures of dissimilarity.
\newblock In F.R. Hodson, D.G. Kendall, and P.~Tautu, editors, {\em Mathematics
  in the Archaeological and Historical Sciences}, pages 387--395. Edinburgh
  University Press, Edinburgh, 1971.

\bibitem[CN05]{MR2197811}
I.~Chatterji and G.~Niblo.
\newblock From wall spaces to {$\rm CAT(0)$} cube complexes.
\newblock {\em Internat. J. Algebra Comput.}, 15(5-6):875--885, 2005.

\bibitem[Gen]{Book}
A.~Genevois.
\newblock Cubulable groups.
\newblock {\em Book in preparation, draft available on the author's webpage}.

\bibitem[Gen17]{QM}
A.~Genevois.
\newblock Cubical-like geometry of quasi-median graphs and applications to
  geometric group theory.
\newblock {\em PhD thesis, Universit\'e Aix-Marseille, arxiv:1712.01618}, 2017.

\bibitem[Gen22]{MR4449680}
A.~Genevois.
\newblock Lamplighter groups, median spaces and {H}ilbertian geometry.
\newblock {\em Proc. Edinb. Math. Soc. (2)}, 65(2):500--529, 2022.

\bibitem[Gen23]{WhyMedian}
A.~Genevois.
\newblock Why {CAT}(0) cube complexes should be replaced with median graphs.
\newblock {\em arxiv:2309.02070}, 2023.

\bibitem[Geo08]{MR2365352}
R.~Geoghegan.
\newblock {\em Topological methods in group theory}, volume 243 of {\em
  Graduate Texts in Mathematics}.
\newblock Springer, New York, 2008.

\bibitem[HM02]{MR1844029}
K.~Huber and V.~Moulton.
\newblock The relation graph.
\newblock volume 244, pages 153--166. 2002.
\newblock Algebraic and topological methods in graph theory (Lake Bled, 1999).

\bibitem[Hou74]{MR357679}
C.~Houghton.
\newblock Ends of locally compact groups and their coset spaces.
\newblock {\em J. Austral. Math. Soc.}, 17:274--284, 1974.

\bibitem[HRS10]{Phylogenetics}
D.~Huson, R.~Rupp, and C.~Scornavacca.
\newblock {\em Phylogenetic Networks: Concepts, Algorithms and Applications}.
\newblock Cambridge University Press, 2010.

\bibitem[KR89]{MR1025923}
P.~Kropholler and M.~Roller.
\newblock Relative ends and duality groups.
\newblock {\em J. Pure Appl. Algebra}, 61(2):197--210, 1989.

\bibitem[LR04]{MR2093872}
J.~Lennox and D.~Robinson.
\newblock {\em The theory of infinite soluble groups}.
\newblock Oxford Mathematical Monographs. The Clarendon Press, Oxford
  University Press, Oxford, 2004.

\bibitem[Mul80]{MR605838}
H.~Mulder.
\newblock {\em The interval function of a graph}, volume 132 of {\em
  Mathematical Centre Tracts}.
\newblock Mathematisch Centrum, Amsterdam, 1980.

\bibitem[Nic04]{MR2059193}
B.~Nica.
\newblock Cubulating spaces with walls.
\newblock {\em Algebr. Geom. Topol.}, 4:297--309, 2004.

\bibitem[Nie79]{MR529794}
J.~Nieminen.
\newblock The ideal structure of simple ternary algebras.
\newblock {\em Colloq. Math.}, 40(1):23--29, 1978/79.

\bibitem[NR98]{MR1459140}
G.~Niblo and M.~Roller.
\newblock Groups acting on cubes and {K}azhdan's property ({T}).
\newblock {\em Proc. Amer. Math. Soc.}, 126(3):693--699, 1998.

\bibitem[Sag95]{MR1347406}
M.~Sageev.
\newblock Ends of group pairs and non-positively curved cube complexes.
\newblock {\em Proc. London Math. Soc. (3)}, 71(3):585--617, 1995.

\bibitem[Sco78]{MR487104}
P.~Scott.
\newblock Ends of pairs of groups.
\newblock {\em J. Pure Appl. Algebra}, 11(1-3):179--198, 1977/78.

\bibitem[vdV93]{MR1234493}
M.~van~de Vel.
\newblock {\em Theory of convex structures}, volume~50 of {\em North-Holland
  Mathematical Library}.
\newblock North-Holland Publishing Co., Amsterdam, 1993.

\bibitem[Wal03]{MR2031877}
C.~Wall.
\newblock The geometry of abstract groups and their splittings.
\newblock {\em Rev. Mat. Complut.}, 16(1):5--101, 2003.

\end{thebibliography}

\Address

%

\end{document}